\documentclass{article}
\usepackage[utf8]{inputenc}
\usepackage[T1]{fontenc}

\usepackage{authblk}
\usepackage{graphicx}
\usepackage{amsmath,amssymb,amsthm}
\usepackage{color}
\usepackage{url}
\usepackage{subfig}
\usepackage{enumitem}

\usepackage[margin=1.3in]{geometry}
\usepackage{longtable}
\usepackage{comment}

\usepackage{tikz,pgf}

\usepackage{mathtools}

\newcommand{\abs}[1]{\left|#1\right|}
\newcommand{\set}[1]{\left\{#1\right\}}

\newcommand{\NN}{\mathbb{N}}

\usetikzlibrary{decorations.markings}
\usetikzlibrary{decorations.pathmorphing}
\usetikzlibrary{patterns}
\tikzset{->-/.style={decoration={
  markings,
  mark=at position .5 with {\arrow[scale=0.8]{>}}},postaction={decorate}}}
\tikzset{snake it/.style={decorate, decoration={snake, amplitude=.4mm, segment length=2mm}}}

\newcommand{\qedclaim}{\hfill $\diamond$ \medskip}
\newenvironment{proofclaim}{\noindent{\em Proof of the claim.}}{\qedclaim}

\usepackage{xcolor}

\newtheorem{theorem}{Theorem}[section]

\newtheorem{observation}[theorem]{Observation}
\newtheorem{claim}[theorem]{Claim}
\newtheorem{corollary}[theorem]{Corollary}
\newtheorem{proposition}[theorem]{Proposition}

\newenvironment{keywords}{
       \list{}{\advance\topsep by0.35cm\relax\small
       \leftmargin=1cm
       \labelwidth=0.35cm
       \listparindent=0.35cm
       \itemindent\listparindent
       \rightmargin\leftmargin}\item[\hskip\labelsep
                                     \bfseries Keywords:]}
     {\endlist}

\newcommand{\chis}{\chi_\Sigma}
\newcommand{\chip}{\chi_\Pi}

\newcommand{\chs}{{\rm ch}_\Sigma}
\newcommand{\chp}{{\rm ch}_\Pi}
\newcommand{\chps}{{\rm ch}_\Pi^*}
\newcommand{\ch}{{\rm ch}}

\newtheorem*{123c}{1-2-3 Conjecture}
\newtheorem*{m123c}{Multiplicative 1-2-3 Conjecture}
\newtheorem*{l123c}{List 1-2-3 Conjecture}
\newtheorem*{lm123c}{List Multiplicative 1-2-3 Conjecture}
\newtheorem*{cn}{Combinatorial Nullstellensatz}

\begin{document}

\title{On a List Variant of the Multiplicative 1-2-3 Conjecture}

\author[1]{Julien Bensmail}
\author[2]{Hervé Hocquard}
\author[2]{Dimitri Lajou}
\author[2]{\'Eric Sopena}

\affil[1]{{\small Universit\'e C\^ote d'Azur, Inria, CNRS, I3S, France}}
\affil[2]{{\small Univ. Bordeaux, CNRS,  Bordeaux INP, LaBRI, UMR 5800, F-33400, Talence, France}}

\maketitle

\begin{abstract}
The 1-2-3 Conjecture asks whether almost all graphs can be (edge-)labelled with $1,2,3$ so that no two adjacent vertices are incident to the same sum of labels.
In the last decades, several aspects of this problem have been studied in literature, including more general versions and slight variations.
Notable such variations include the List 1-2-3 Conjecture variant, in which edges must be assigned labels from dedicated lists of three labels,
and the Multiplicative 1-2-3 Conjecture variant, in which labels~$1,2,3$ must be assigned to the edges so that adjacent vertices are incident to different products of labels.
Several results obtained towards these two variants led to observe some behaviours that are distant from those of the original conjecture.   

In this work, we consider the list version of the Multiplicative 1-2-3 Conjecture, proposing the first study dedicated to this very problem.
In particular, given any graph $G$, we wonder about the minimum~$k$ such that $G$ can be labelled as desired when its edges must be assigned labels from dedicated lists of size~$k$.
Exploiting a relationship between our problem and the List 1-2-3 Conjecture, we provide upper bounds on~$k$ when $G$ belongs to particular classes of graphs.
We further improve some of these bounds through dedicated arguments. 

\end{abstract}

{}

\section{Introduction}

Let $G$ be a graph and $\ell$ be a \textit{$k$-labelling} of $G$, i.e., an assignment $\ell : E(G) \rightarrow \{1,\dots,k\}$ of labels $1,\dots,k$ to the edges of $G$.
For every vertex $v$ of $G$, one can compute, as a colour, the \textit{sum} $\sigma_\ell(v)$ of labels assigned by $\ell$ to the edges incident to $v$, that is $$\sigma_\ell(v)= \sum_{w \in N(v)} \ell(vw).$$
We say that $\ell$ is \textit{s-proper} if $\sigma_\ell$ is a proper vertex-colouring of $G$, i.e., if, for every edge $uv$ of $G$, we have $\sigma_\ell(u) \neq \sigma_\ell(v)$.
We denote by $\chis(G)$ the smallest $k \geq 1$, if any, such that $G$ admits s-proper $k$-labellings.
It turns out that $\chis(G)$ is defined, i.e., that $G$ admits s-proper labellings, if and only if $G$ has no connected component isomorphic to $K_2$.
For this reason, when investigating s-proper labellings, we generally focus on so-called \textit{nice graphs}, which are those graphs with no connected component isomorphic to $K_2$,
i.e., having their parameter $\chis$ being properly defined.

The \textbf{1-2-3 Conjecture}, introduced in~\cite{KLT04} by Karo\'nski, {\L}uczak and Thomason in 2004, presumes that the maximum value of $\chis(G)$ for a nice graph $G$
should never exceed~$3$; that is:

\begin{123c}
If $G$ is a nice graph, then $\chis(G) \leq 3$.
\end{123c}

Several aspects towards this conjecture have been investigated to date.
For an in-depth review of most of our knowledge on the problem,
we refer the reader to the survey~\cite{Sea12} by Seamone.
Let us mention, as notable evidence towards the 1-2-3 Conjecture, that it is known to hold for $3$-colourable graphs~\cite{KLT04}, and that $\chis(G) \leq 5$ holds for every nice graph $G$~\cite{KKP10}.

\medskip

Our investigations in this work are primarily related to two variants of the 1-2-3 Conjecture,
being its \textbf{Multiplicative} and \textbf{List variants},
which we recall in what follows.

As the name suggests, the Multiplicative variant is related to products of labels rather than to sums of labels.
The terminology is as follows. Let $G$ be a graph and $\ell$ be a labelling of $G$.
This time, for every vertex $v$ of $G$, we compute, as a colour, the \textit{product} $\pi_\ell(v)$ of labels incident to $v$ by $\ell$,
that is $$\pi_\ell(v)= \prod_{w \in N(v)} \ell(vw).$$
We say that $\ell$ is \textit{p-proper} if $\pi_\ell$ is a proper vertex-colouring of $G$, i.e., if, for every edge $uv$ of $G$, we have $\pi_\ell(u) \neq \pi_\ell(v)$.
Assuming $G$ is nice, we denote by $\chip(G)$ the smallest $k$ such that $G$ admits p-proper $k$-labellings.

Introduced in~\cite{SK12} by Skowronek-Kazi\'ow in 2012, the \textbf{Multiplicative 1-2-3 Conjecture} stands as the straight product counterpart to the original 1-2-3 Conjecture:

\begin{m123c}
If $G$ is a nice graph, then $\chip(G) \leq 3$.
\end{m123c}

The Multiplicative 1-2-3 Conjecture has received, to date, much less attention than the original 1-2-3 Conjecture.
Towards it, the main results we know of are that the conjecture holds for $4$-colourable graphs~\cite{BHLS20},
and that $\chip(G) \leq 4$ holds for every nice graph $G$~\cite{SK12}.

The List variant of the 1-2-3 Conjecture is a generalisation where edges must be assigned labels from fixed-size lists that might be different from $\{1,2,3\}$.
This is all defined accordingly to the following definitions.
Let $G$ be a graph and $L$ be a \textit{$k$-list assignment} (to the edges) of $G$, i.e., an assignment of sets of $k$ real numbers to the edges. 
An \textit{$L$-labelling} $\ell$ of $G$ is a labelling where each edge is assigned a label from its list, i.e., $\ell(e) \in L(E)$ for every $e \in E(G)$.
Note that the notion of s-proper labellings can be extended naturally to $L$-labellings.
Now, for a nice graph $G$, we define $\chs(G)$ as the smallest $k$ such that $G$ admits an s-proper $L$-labelling for every $k$-list assignment $L$.

The \textbf{List 1-2-3 Conjecture}, introduced in 2009 by Bartnicki, Grytczuk and Niwczyk in~\cite{BGN09}, 
is the straight analogue of the 1-2-3 Conjecture to the previous notions:

\begin{l123c}
If $G$ is a nice graph, then $\chs(G) \leq 3$.
\end{l123c}

The List 1-2-3 Conjecture is of course much stronger than the original conjecture, and, as a matter of fact,
there is still no known general constant upper bound on $\chs$. To date, the best bound we know of, is that $\chs(G) \leq \Delta(G)+1$ holds for every nice graph $G$~\cite{DDWWWYZ19} (where $\Delta(G)$ denotes the maximum degree of $G$).
Constant upper bounds were established for some classes of graphs; see later Section~\ref{section:cn-and-connections} for more details.
For now, let us just mention that most of these results were established through an approach imagined by Bartnicki, Grytczuk and Niwczyk in~\cite{BGN09},
which is reminiscent of the studies on choosability of graphs, which relies on attacking the problem from an algebraic point of view.
This will also be described further in a later section (Subsection~\ref{subsection:algebraic-tools}), as this is an important point behind our investigations.
	
\medskip

This work is dedicated to investigating a new problem inspired from the previous ones above, 
holding, essentially, as a \textbf{List Multiplicative 1-2-3 Conjecture}.
Note that, for an $L$-labelling of a nice graph $G$, the notion of p-properness adapts naturally,
and from this we can define $\chp(G)$ as the smallest $k$ such that $G$ admits a p-proper $L$-labelling for every $k$-list assignment $L$ to its edges.
This parameter $\chp$ is precisely the one we study throughout this work. 
To the best of our knowledge, this parameter was, to date, only discussed briefly by Seamone in his survey~\cite{Sea12}, in which he suggests a few of its properties. This parameter is also somewhat close to other studied parameters, such as the notions of \textit{product irregularity strength}~\cite{Anh09} (related to labellings for which all vertices, not only the adjacent ones, must be incident to distinct products of labels) and \textit{neighbour-product-distinguishing index}~\cite{LQWY17} (related to labellings for which the labels assigned to the edges must form a proper edge-colouring).

The current paper consists of two main sections.
Section~\ref{section:cn-and-connections} stands as a preliminary section in which we raise first observations on the parameter $\chp$.
In that section, we also explore the connections between our problem and the List 1-2-3 Conjecture,
from which we get first systematic upper bounds on $\chp$.
We also get to describing the algebraic approach through which we improve some of these first bounds.
These improved bounds are gathered in Section~\ref{section:improved-bounds}, and are about both general graphs (Subsection~\ref{subsection:general-bounds}) and particular classes of graphs, such as trees, planar graphs with large girth, and subcubic graphs (Subsection~\ref{subsection:bounds-particular-classes}).


\section{Preliminaries, tools, and connections with the sum variant}\label{section:cn-and-connections}

We here introduce all tools and preliminary materials needed to establish the results in later Section~\ref{section:improved-bounds}. In Subsection~\ref{subsection:remarks-chp}, we first state a few easy observations on the parameter $\chp$. In Subsection~\ref{subsection:connections-sum-variant}, we establish and exploit a relationship between the two parameters $\chs$ and $\chp$, from which we deduce first constant bounds on $\chp$ for several graph classes. Finally, in Subsection~\ref{subsection:algebraic-tools}, we recall algebraic tools from which improved bounds will be obtained, later in Section~\ref{section:improved-bounds}.

For transparency, let us mention that some of the results from this section, mostly from Subsection~\ref{subsection:remarks-chp}, were already suggested by Seamone in~\cite{Sea12}, in which the parameter $\chp$ is discussed very briefly. To make our contribution clear, we notify properly, through what follows, every remark also mentioned in~\cite{Sea12}.


\subsection{Early remarks on the parameter $\chp$}\label{subsection:remarks-chp} 

As remarked in~\cite{Sea12}, note, given an edge $uv$ of a graph $G$, that if $\ell(uv)=0$ by a labelling $\ell$ of $G$, then $\ell$ cannot be p-proper, since this would imply $\pi_\ell(u)=\pi_\ell(v)=0$. Thus, for any list assignment $L$ of $G$, a p-proper $L$-labelling is actually a p-proper $L^*$-labelling, where $L^*$ is the list assignment of $G$ verifying $L^*(e)=L(e) \setminus \{0\}$ for every edge $e \in E(G)$. 
Therefore, throughout this work, we consider list assignments not assigning label~$0$ to the edge lists. To catch this point, we refine the parameter $\chp(G)$ of a graph $G$ to the parameter $\chps(G)$, which is the smallest $k \geq 1$, if any, such that $G$ admits p-proper $L$-labellings for every $k$-list assignment $L$ not assigning label~$0$. 

By the previous remarks, obviously the following holds:

\begin{observation}\label{observation:nozero}
If $G$ is a nice graph, then $\chp(G)=\chps(G)+1$.
\end{observation}

We note that if $L$ is the $1$-list assignment of $G$ where $L(e) = \{1\}$ for every edge $e$, then $G$ admits no p-proper $L$-labellings, since every such labelling $\ell$ would verify $\pi_\ell(u)=\pi_\ell(v)=1$ for every edge $uv \in E(G)$. Thus:

\begin{observation}
There is no graph $G$ verifying $\chps(G)=1$.
\end{observation}

Analogous conclusions can be reached regarding graphs $G$ with $\chps(G)=2$. Here, consider the $2$-list assignment $L$ of $G$ where $L(e)=\{-1,1\}$ for every edge $e$. Then, by an $L$-labelling $\ell$ of $G$, we have $\pi_\ell(v) \in \{-1,1\}$ for every vertex $v \in V(G)$. This implies that $\ell$ is p-proper if and only if $\pi_\ell$ is a proper $2$-vertex-colouring of $G$. In turn, this yields the following (also mentioned in~\cite{Sea12}):

\begin{observation}
If $G$ is a graph with $\chps(G)=2$, then $G$ is bipartite.
\end{observation}

The previous condition is not sufficient, however, as nice connected bipartite graphs $G$ with $\chps(G)=2$ must fulfil an additional property.

\begin{proposition}
Let $G$ be a connected bipartite graph with bipartition $(A,B)$. If $\chps(G)=2$, then at least one of $|A|$ and $|B|$ must be even.
\end{proposition}

\begin{proof}
Assume the claim is wrong, and let $G$ be a connected bipartite graph with $\chps(G)=2$ in which the two parts $A$ and $B$ are of odd size. Consider $L$, the $2$-list assignment of $G$ where $L(e)=\{-1,1\}$ for every edge $e \in E(G)$. As mentioned earlier, by every $L$-labelling $\ell$ of $G$, we have $\pi_\ell(v) \in \{-1,1\}$ for every vertex $v \in V(G)$. Thus, because $G$ is connected, for such an $\ell$ to be p-proper we must have, say, $\pi_\ell(a) = -1$ for every $a \in A$ and $\pi_\ell(b) = 1$ for every $b \in B$. For the first condition to occur, for every $a \in A$ there must be an odd number of incident edges labelled $-1$ by $\ell$. Since $|A|$ is odd, this means that we must have an odd number of edges of $G$ labelled~$-1$ by $\ell$. For the second condition to occur, for every $b \in B$ there must be an even number of incident edges labelled $-1$ by $\ell$. 
For that, we must have an even number of edges of $G$ labelled~$-1$ by $\ell$, which is a contradiction.
\end{proof}

Thus, connected graphs $G$ with $\chps(G)=2$ are connected bipartite graphs with at least one part of even cardinality. This condition is necessary but still not sufficient, however, even in simple graph classes such as trees. To see this is true, consider the following easy remarks. 

Suppose we have a graph $G$ with a pending path $wvu$ of length~$2$, where $d(u)=1$ and $d(v)=2$, and suppose $L$ is a $2$-list assignment to the edges of $G$. Assume more particularly that $L(wv)=\{1,a\}$ for some $a \neq 1$. Then, note that, in any p-proper $L$-labelling $\ell$ of $G$, we cannot have $\ell(vw)=1$, as otherwise we would have $\pi_\ell(v)=\pi_\ell(u)$ whatever $\ell(vu)$ is, a contradiction. In other words, the label of $wv$ by a p-proper $L$-labelling of $G$ is forced to $a$.

From this, we can construct arbitrarily many trees $T$ with $\chps(T)=3$ and any wanted cardinality parity for the parts of its bipartition. As an illustration (which admits obvious generalisations), consider the tree $T$ with vertex set $V(T)=\{v_1,\dots,v_8\}$ and edge set $E(T)=\{v_1v_2,v_2v_5,v_3v_4,v_4v_5,v_5v_6,v_6v_7,v_7v_8\}$, and note that $T$ has no p-proper $L$-labelling for any list assignment $L$ where $L(v_6v_7)=\{1,a^2\}$ and $L(v_2v_5)=L(v_4v_5)=\{1,a\}$ (for some $a \not \in \{1,-1\}$).


\subsection{Connections with the sum variant, and first bounds on $\chp$}\label{subsection:connections-sum-variant}

As suggested by Seamone in~\cite{Sea12}, there is a straight connection between the parameters $\chs$ and $\chps$, which follows from the product rule of logarithms.
Despite this fact being easy to visualise, we give a detailed proof to establish the precise relationship between the two.

\begin{theorem}\label{theorem:chs-to-chp}
If $G$ is a nice graph, then $\chps(G) \leq 2\chs(G)-1$.
\end{theorem}

\begin{proof}
Assume we have $\chs(G) \leq k$ for some nice graph $G$ and $k \geq 2$. We prove that $\chps(G) \leq 2k-1$. Let $L$ be a $(2k-1)$-list assignment to the edges of $G$, where none of the $L(e)$'s contains label~$0$. For every $e \in E(G)$, since $|L(e)|=2k-1$,  there must be $S(e) \subset L(e)$ such that $|S(e)|=k$ and no two elements of $S(e)$ have the same absolute value. We set $X(e) = \{|x| : x \in S(e)\}$ and $L'(e) = \{\log(x) : x \in X(e)\}$\footnote{Throughout this work, any used $\log$ function can be in any fixed base.}. Then $L'$ is a $k$-list assignment of $G$ where each edge $e$ is associated $k$ nonnegative values that are logarithms of values of $L(e)$ with different absolute values.

Our original assumption $\chs(G) \leq k$ implies that $G$ admits an s-proper $L'$-labelling $\ell'$. We now consider an $L$-labelling $\ell$ of $G$ obtained as follows. We consider every edge $e$ of $G$, and we choose, as $\ell(e)$, any label from $L(e)$ that resulted in $L'(e)$ containing $\ell'(e)$. By how $L'$ was obtained, note that, indeed, one such value belongs to $L(e)$. Thus, $\ell$ is an $L$-labelling. As a result, for every $v \in V(G)$ with incident edges $e_1, \dots, e_d$, we get $$\sigma_{\ell'}(v)=\sum_{i=1}^d \ell'(e_i)=\sum_{i=1}^d \left(\log |\ell(e_i)|\right) =\log\left(\prod_{i=1}^d |\ell(e_i)|\right)= \log(\abs{\pi_\ell(v)}).$$ 
In particular, $\ell$ is p-proper since $\ell'$ is s-proper.
\end{proof}

The connection between $\chs$ and $\chps$ in Theorem~\ref{theorem:chs-to-chp} implies that, for any constant upper bound on $\chs$ for some graph class, we deduce a constant upper bound on $\chps$ as well. In the next result, we have listed some constant bounds on $\chs$ from the literature, together with the bounds on $\chps$ we get as a consequence. It is worth emphasising that we do not claim this list to be exhaustive in any way. Namely, we only list the bounds that seem the most significant to us, and the interested reader has to be aware that more results of the sort below can be established from results mentioned in the references below. 

\begin{corollary}\label{corollary:existing-bounds}
Let $G$ be a nice connected graph.
\begin{itemize}
    \item $\chs(G) \leq \Delta(G)+1$ (see \cite{DDWWWYZ19}); thus $\chps(G) \leq 2\Delta(G)+1$.
    
    \item If $G$ is complete, complete bipartite, or a tree, then $\chs(G) \leq 3$ (see \cite{BGN09}); thus $\chps(G) \leq 5$.

    \item If $G$ is $2$-degenerate and non-bipartite, then $\chs(G) \leq 3$ (see \cite{WZ18}); thus $\chps(G) \leq 5$.
    
    \item If $G$ is a wheel, then $\chs(G) \leq 3$ (see \cite{PY13}); thus $\chps(G) \leq 5$.
    
    \item If ${\rm mad}(G) \leq \frac{11}{4}$, then $\chs(G) \leq 3$ (see \cite{LWZ18}); thus $\chps(G) \leq 5$.
    
    \item If $G$ is outerplanar, then $\chs(G) \leq 4$ (see \cite{PY13}); thus $\chps(G) \leq 7$.
    
    \item If $\Delta(G) \leq 4$, then $\chs(G) \leq 4$ (see \cite{LLM20}); thus $\chps(G) \leq 7$.
    
    \item If $G$ is $2$-connected and chordal, or a line graph, then $\chs(G) \leq 5$ (see \cite{Won21}); thus $\chps(G) \leq 9$.
    
    \item If $G$ is a planar graph, then $\chs(G) \leq 7$ (see \cite{WZ18}); thus $\chps(G) \leq 13$.
\end{itemize}
\end{corollary}

A consequence of the first item in Corollary~\ref{corollary:existing-bounds}, is that the List 1-2-3 Conjecture itself makes plausible the existence of a general constant upper bound on $\chps$.
In particular, we currently have no evidence that the following, which would be a legitimate guess, might be false:

\begin{lm123c}
If $G$ is a nice graph, then $\chps(G) \leq 3$.
\end{lm123c}

Recall that observations raised at the end of Subsection~\ref{subsection:remarks-chp} establish that this conjecture, if true, would actually be tight. 


\subsection{Algebraic tools}\label{subsection:algebraic-tools}

To improve, in next Section~\ref{section:improved-bounds}, some of the bounds from Corollary~\ref{corollary:existing-bounds}, we will adapt and employ an algebraic approach that was first designed by Bartnicki, Grytczuk and Niwczyk to deal with the List 1-2-3 Conjecture in~\cite{BGN09}, and which is inspired by polynomial methods developed to deal with list colouring of graphs.

Consider a graph $G$ with edges $e_1,\dots,e_m$, and a list assignment $L$ to the edges of $G$.
For a vertex $u$ and an edge $e$ of $G$, we write $e \sim u$ if $e$ is incident to $u$.
Let $\vec{G}$ be any orientation of $G$.
To each edge $e_i$ of $G$, we associate a variable $x_i$.
Now, we associate to $G$ (through $\vec{G}$) a polynomial $Q_{\vec{G}}$ with variables $x_1,\dots,x_m$, being 
$$Q_{\vec G}(x_1, \dots x_m)= \prod_{\vec{uv} \in A(\vec{G})} \left(\sum_{e_i \sim u} x_i - \sum_{e_i \sim v} x_i\right).$$
It is easy to see that $G$ has an s-proper $L$-labelling if and only if there are values $l_1 \in L(e_1),\dots,l_m \in L(e_m)$ such that $Q_{\vec{G}}(l_1,\dots,l_m)$ does not vanish. From this point of view, a powerful tool is the so-called \textit{Combinatorial Nullstellensatz} of Alon~\cite{Alon99}, which provides sufficient conditions, in terms of the sizes of the lists $L(e_1),\dots,L(e_m)$, for such values $l_1,\dots,l_m$ to be choosable.

\begin{cn}
Let $\mathbb{F}$ be an arbitrary field, and let $f=f(x_1,\dots,x_n)$ be a polynomial in $\mathbb{F}[x_1,\dots,x_n]$.
Suppose the total degree of $f$ is $\sum_{i=1}^n t_i$, where each $t_i$ is a nonnegative integer, and suppose the coefficient of $\prod_{i=1}^n x_i^{t_i}$ is nonzero.
If $S_1, \dots, S_n$ are subsets of $\mathbb{F}$ with $|S_i|>t_i$, then there are $s_1 \in S_1, \dots, s_n \in S_n$ so that $f(s_1,\dots,s_n) \neq 0$.
\end{cn}

Thus, bounds on $\chs(G)$ can be obtained via the Combinatorial Nullstellensatz through studying the monomials in the expansion of $Q_{\vec{G}}$, more precisely monomials with nonzero coefficient, maximum degree, and, preferably, low exponent values.
Note that all the monomials of $Q_{\vec{G}}$ share the very convenient property that they are all of maximum degree $m$, which is one of the prerequisites for the Combinatorial Nullstellensatz to work. The tricky part, actually, is about anticipating the coefficients of the monomials of $Q_{\vec{G}}$ (the nonzero ones, particularly), which are far from being obvious in general. In~\cite{BGN09}, the authors developed a very nice dedicated approach, which is based on studying the permanent of a particular matrix representing $Q_{\vec{G}}$.

\medskip

A similar polynomial approach can of course be applied for deducing bounds on $\chp(G)$.
The main difference is that, this time, we have to consider the products of labels incident to the vertices, instead of their sums.
More precisely, the polynomial of interest is here
$$P_{\vec G}(x_1, \dots x_m)= \prod_{\vec{uv} \in A(\vec{G})} \left(\prod_{e_i \sim u} x_i - \prod_{e_i \sim v} x_i\right).$$
Compared to the polynomial $Q_{\vec{G}}$, a big difference is that, in the expansion of $P_{\vec{G}}$, the monomials are likely to have different degrees, which means that the Combinatorial Nullstellensatz might apply to a few of them only. Even worse is that the degree of $P_{\vec{G}}$ is generally bigger than that of $Q_{\vec{G}}$, and, in particular, the exponents of the monomials generally tend to be bigger too. Note indeed that the degree of $Q_{\vec{G}}$ is precisely $m$, while the degree of $P_{\vec{G}}$ can be as large as $\sum_{uv \in E(G)} \max\{d(u),d(v)\}$ (which can be reached, e.g. when no two adjacent vertices of $G$ have the same degree). 

For these reasons, as will be seen in next Section~\ref{section:improved-bounds}, deducing bounds on $\chps$ via the Combinatorial Nullstellensatz only, seems to be viable in particular contexts only.

\section{Improved bounds on $\chps$ for some graph classes} \label{section:improved-bounds}

We here improve some of the bounds on $\chps$ from Corollary~\ref{corollary:existing-bounds}. We first consider graphs in general, in Subsection~\ref{subsection:general-bounds}. We then focus, in Subsection~\ref{subsection:bounds-particular-classes}, on particular classes of graphs, including trees, planar graphs with large girth, and subcubic graphs. In the latter subsection, the exhibited improved bounds are optimal, or close to optimal.

\subsection{General graphs}\label{subsection:general-bounds}

The bounds on $\chps$ we establish in this section are expressed as functions of the maximum degree, our goal being to improve the bound of the first item of Corollary~\ref{theorem:chs-to-chp}.
We start off by improving that bound slightly for all nice graphs. From the bound we provide, we deduce, towards the List Multiplicative 1-2-3 Conjecture, that the List 1-2-3 Conjecture, if verified, would imply that $\chps(G) \leq 5$ holds for every nice graph $G$.

\begin{theorem}\label{theorem:upper-bound-cn}
If $G$ is a nice graph, then $\chps(G) \leq 2\Delta(G)-1$.
\end{theorem}

\begin{proof}
Let us denote by $e_1,\dots,e_m$ the edges of $G$, and, for every $i \in \{1,\dots,m\}$, let $x_i$ be a variable associated to $e_i$. Now, let $\vec{G}$ be any orientation of $G$, and $P_{\vec{G}}$ be the polynomial with variables $x_1,\dots,x_m$ defined as 
$$P_{\vec G}(x_1, \dots x_m)= \prod_{\vec{uv} \in A(\vec{G})} \left(\prod_{e_i \sim u} x_i - \prod_{e_i \sim v} x_i\right).$$
As described in Subsection~\ref{subsection:algebraic-tools},
if $L$ is a list assignment of $G$, and $P_{\vec{G}}(l_1,\dots,l_m) \neq 0$ for some $l_1 \in L(e_1),\dots,l_m\in L(e_m)$, then clearly we deduce a p-proper $L$-labelling of $G$.

Let $M=cx_1^{t_1}\dots x_m^{t_m}$ be a monomial of maximum degree from the expansion of $P_{\vec{G}}$ with $c \neq 0$.
Such an $M$ has to exist, since $\mathbb{R}[x_1,\dots,x_m]$ is an integral domain.
We note that for every $i \in \{1,\dots,m\}$, we have $t_i \leq 2\Delta(G)-1$. This is because variable $x_i$ appears in at most $2\Delta(G)-1$ factors of $P_{\vec{G}}$: once due to the edge $e_i$, and at most $2\Delta(G)-2$ times dues to the other edges incident to the two ends of $e_i$.
By earlier arguments, the Combinatorial Nullstellensatz, due to the existence of $M$,  now implies that $\chp(G) \leq 2\Delta(G)$, thus our conclusion on $\chps$ by Observation~\ref{observation:nozero}. 
\end{proof}

The next bound is a significant improvement over Theorem~\ref{theorem:upper-bound-cn}, in the case of graphs having vertices with convenient neighbourhood properties.

\begin{theorem}\label{theorem:better-bound-triangle-free}
If $G$ is a nice graph with a vertex $u$ such that $d(u) \geq 2$, $N(u)$ is a stable set, and $\chps(G-u) \leq \Delta(G-u)+3$, then $\chps(G) \leq \Delta(G-u)+3$.
\end{theorem}

\begin{proof}
Let $L$ be a $(\Delta(G-u)+3)$-list assignment to the edges of $G$, and let $L'$ be the restriction of $L$ to the edges of $G'=G-u$. Since $\chps(G') \leq \Delta(G')+3$, there is a p-proper $L'$-labelling $\ell'$ of $G'$. Our aim is to extend $\ell'$ to a p-proper $L$-labelling $\ell$ of $G$, by considering the edges $uv_1,\dots,uv_d$ ($d \geq 2$) incident to $u$, and, for each one $uv_i$ of them, assigning it a label from $L(uv_i)$ so that, eventually, no  conflict arises. For every $i \in \{1, \dots, d\}$, we denote by $z_i$ the current product of $v_i$ (i.e., by $\ell'$).

Because $N(v)$ is a stable set, note that assigning a label to any $uv_i$ completely determines the product of $v_i$, in the sense that all edges incident to $v_i$ get labelled. Thus, when labelling $uv_i$, we must ensure that $v_i$ does not get in conflict with its neighbours different from $u$. Since $|N(v_i) \setminus \{u\}| \leq \Delta(G-u)$, and because $|L(uv_i)| = \Delta(G-u)+3$, there are at least three distinct values in $L(uv_i)$ that can be assigned to $uv_i$ without causing any  conflict between $v_i$ and its neighbours different from $u$. For every $i \in \{1,\dots,d\}$, we denote by $S_i$ this subset of ``safe'' values of $L(uv_i)$ for $uv_i$. Because $|S_i| \geq 3$, there are in $S_i$ at least two values $a_i,b_i$ such that $|a_i| \neq |b_i|$. 

We will be done if we can find an assignment of $a_i$'s and $b_i$'s to the $uv_i$'s for which $u$ gets in  conflict with none of the $v_i$'s. Such an assignment is actually guaranteed to exist by the Combinatorial Nullstellensatz. Indeed, for every $i \in \{1, \dots, d\}$, let $x_i$ be a variable associated to the edge $uv_i$. Consider the polynomial $$P(x_1,\dots,x_d)=\prod_{i=1}^d \left(\prod_{j = 1}^d x_j - x_iz_i\right).$$
For every $i \in \{1, \dots, d\}$, set $y_i=\log(|x_i|)$.
Note that considering $P$ is similar to considering $$P'(y_1,\dots,y_d)=\prod_{i=1}^d \left(\sum_{\substack{j=1 \\ j\neq i}}^d y_j - \log(z_i)\right),$$ which, because the $z_i$'s are constants, in the current context is similar to studying $$P''(y_1,\dots,y_d)=\prod_{i=1}^d \sum_{\substack{j=1 \\ j\neq i}}^d y_j.$$ We remark that the monomial $y_1 \dots y_d$ in the expansion of $P''$ has strictly positive coefficient. Thus, by the Combinatorial Nullstellensatz, we can assign values to the $y_i$'s so that $P''$ (and thus $P'$) does not vanish, assuming we have at least two possible values to choose from for each of them. From this, we get that we can assign values to the $x_i$'s so that $P$ does not vanish as long as we have at least two possible values with distinct absolute values to choose from, for each of them. Particularly, this means that we can assign a label from $\{a_i,b_i\}$ to every $uv_i$, in such a way that a p-proper $L$-labelling of $G$ results.
\end{proof}

Together with checking a few base cases (which is done through results in next Subsection~\ref{subsection:bounds-particular-classes}), Theorem~\ref{theorem:better-bound-triangle-free} implies the following:

\begin{corollary}
If $G$ is a nice triangle-free graph, then $\chps(G) \leq \Delta(G)+3$.
\end{corollary}

\subsection{Particular classes of graphs}\label{subsection:bounds-particular-classes}

\subsubsection*{Paths and cycles}

Note that Theorem~\ref{theorem:upper-bound-cn} implies that the List Multiplicative 1-2-3 Conjecture holds for nice graphs $G$ with $\Delta(G) \leq 2$, i.e., paths and cycles. In such simple cases, this can actually be refined to a tightest result. In the sequel, for an $n \geq 2$ (or an $n \geq 3$ in the case of a cycle), we denote by $P_n$ and $C_n$ the path and cycle, respectively, of length~$n$.

\begin{theorem}\label{theorem:cycles}
For an $n \geq 2$, we have:
\begin{itemize}
    \item $\chps(P_n)=2$ if $n$ is even or $n=3$; 
    \item $\chps(P_n)=3$ otherwise.
\end{itemize}
For an $n \geq 3$, we have:
\begin{itemize}
    \item $\chps(C_n)=2$ if $n \equiv 0 \bmod 4$; \item $\chps(C_n)=3$ otherwise.
\end{itemize}
\end{theorem}

\begin{proof}
We deal with cycles first.
Let us denote by $e_0,\dots,e_{n-1}$ the successive edges of $C_n$, and by $v_0,\dots,v_{n-1}$ its successive vertices, where $e_i=v_iv_{i+1}$ for every $i \in \{0,\dots,n-1\}$ (where, here and further, operations over the indexes are understood modulo~$n$). For any two adjacent vertices $v_i$ and $v_{i+1}$, note that, in order to get $\pi_\ell(v_i) \neq \pi_\ell(v_{i+1})$ by a labelling $\ell$ of $C_n$, we must have $\ell(v_{i-1}v_i) \neq \ell(v_{i+1}v_{i+2})$. Thus, for $\ell$ to be p-proper, any two edges of $C_n$ at distance~$2$ apart must be assigned different labels.

Now consider $G$, the graph constructed from $C_n$ by adding one vertex $v_{e_i}$ in $G$ for every edge $e_i$ of $C_n$, and adding an edge $v_{e_i}v_{e_j}$ between any two vertices $v_{e_i},v_{e_j}$ of $G$ if $e_i$ and $e_j$ are at distance exactly~$2$ in $C_n$. By a remark above, we have $\chps(C_n)=\ch(G)$ (where $\ch(G)$ refers to the usual choice number of $G$). Note that $G$ is an odd-length cycle when $n$ is odd, an union of two odd-length cycles when $n \equiv 2 \bmod 4$, and an union of two even-length cycles when $n \equiv 0 \bmod 4$. Since even-length cycles have choice number~$2$ and odd-length cycles have choice number~$3$ (see e.g.~\cite{ERT79}), the result follows.

\medskip

Regarding paths, remark first that if $n \equiv 1 \bmod 4$, then $P_n$ is a bipartite graph in which the two parts of the bipartition have odd cardinality. As described at the end of Subsection~\ref{subsection:remarks-chp}, we must have $\chps(P_n)>2$ in such a situation, and we actually have $\chps(P_n) = 3$ by Theorem~\ref{theorem:upper-bound-cn}. 

Let us now consider the remaining values of $n$.
For a given $n \geq 2$, similarly as in the case of cycles, let us denote by $e_1, \dots, e_n$ the successive edges of $P_n$, and by $v_1, \dots, v_{n+1}$ its vertices, where $e_i=v_iv_{i+1}$ for every $i \in \{1,\dots,n\}$. Note that, contrarily to the case of cycles, labelling $P_n$ is not similar to colouring $G$, the constraint graph of the edges at distance~$2$ in $P_n$, because, when labelling $P_n$, we must also guarantee that $e_2$ and $e_{n-1}$ are not assigned label~$1$, so that $v_1$ and $v_2$ and not in conflict, and similarly for $v_n$ and $v_{n+1}$. Note that $G$ is here the union of two (possibly empty) paths; the new labelling constraint is similar to forbidding one colour for each of $v_{e_2}$ and $v_{e_{n-1}}$. A problem is when these two vertices are the ends of a same path of $G$.

Indeed, from this remark, we note that in the cases where  $n \equiv 3 \bmod 4$ as well, there are $2$-list assignments $L$ of $P_n$ such that $P_n$ has no p-proper $L$-labelling. Indeed, note that the set of edges $\{e_2,e_4,\dots,e_{n-1}\}$ has odd cardinality due to the value of $n$. Recall that every two edges at distance~$2$ in $P_n$ in this set must receive distinct labels by a p-proper labelling. Then note that if $L(e_2)=\{1,a\}$ for some $a \not \in \{-1,1\}$ and $L(e_{2k})=\{a,b\}$ for some $b \not \in \{1,a\}$ for every other edge $e_{2k} \not \in \{e_2,e_{n-1}\}$, then, depending on the value of $n$, for either $L(e_{n-1})=\{1,a\}$ or $L(e_{n-1})=\{1,b\}$ label~$1$ must be assigned to one of $e_2$ and $e_{n-1}$ by a p-proper $L$-labelling, which creates a conflict, a contradiction. Then $\chps(P_n)>2$ for such a value of $n$, and we have $\chps(P_n) = 3$ by Theorem~\ref{theorem:upper-bound-cn}. 

Let us now consider the remaining cases, i.e., those where $n=3$ or $n$ is even. Let $L$ be a $2$-list assignment of $P_n$. We deduce a p-proper $L$-labelling $\ell$ of $P_n$ is the following way:
\begin{itemize}
    \item If $n=3$, then first assign to $e_2$ a label from $L(e_2)$ different from~$1$, before assigning distinct labels from $L(e_1)$ and $L(e_3)$ to $e_1$ and $e_3$, respectively.
    Clearly, $\ell$ is p-proper.
    
    \item Assume $n$ is even. If $n=2$, then, clearly, we are done when assigning labels from $L(e_1)$ and $L(e_2)$ different from~$1$ to $e_1$ and $e_2$, respectively. So assume $n \geq 4$. We first label the edges $e_1,e_3,\dots,e_{n-1}$ with odd index with labels from their respective lists, in such a way that 1) $\ell(e_{n-1}) \neq 1$, and that 2) no two of these edges at distance~$2$ are assigned the same label. These conditions can clearly be achieved by labelling these edges one by one following the ordering $e_{n-1}, e_{n-3}, \dots, e_1$. We then achieve the same thing for the edges $e_2,e_4,\dots,e_n$ with even index, so that 1) $\ell(e_2) \neq 1$, and that 2) no two of these edges at distance~$2$ are assigned a same label. Again, this can be easily achieved, e.g. by labelling these edges following the ordering $e_2,e_4,\dots,e_n$. By  arguments above, $\ell$ is eventually p-proper. \qedhere 
\end{itemize}
\end{proof}


\subsubsection*{Trees}

We now prove an upper bound on $\chps$ in the case of trees.
The exhibited bound is optimal in general, due to some of the remarks at the end of Subsection~\ref{subsection:remarks-chp}. Even some paths attain the upper bound, recall Theorem~\ref{theorem:cycles}.

\begin{theorem}\label{theorem:trees}
If $T$ is a nice tree, then $\chps(T) \leq 3$.
\end{theorem}

\begin{proof}
The proof is by induction on the number of vertices and edges of $T$. The base case is when $T$ is a path of length~$2$, in which situation the claim holds by Theorem~\ref{theorem:cycles}. Thus, we can focus on proving the general case. Let $L$ be a $3$-list assignment to the edges of $T$.

We can assume that $T$ has no pending path of length at least~$3$, i.e., a path $uvwx$ such that $d(u)=1$, $d(v)=d(w)=2$, and $d(x) \geq 2$. Indeed, assume $T$ has such a path. Let $T'=T-\{u,v\}$. Clearly $T'$ is nice (as otherwise $T$ would be a path, a case for which Theorem~\ref{theorem:cycles} yields the desired conclusion), and thus $T'$ admits a p-proper $L'$-labelling $\ell'$, where $L'$ denotes the restriction of $L$ to the edges of $T'$. To extend $\ell'$ to a p-proper $L$-labelling of $T$, we  have to assign to $uv$ and $vw$ labels from their lists, so that no conflict arises. To that aim, we first assign to $vw$ a label different from 
$1$ 
and from $\frac{\pi_{\ell'(x)}}{\ell'(xw)}$ so that $w$ does not get in conflict with $x$. 
Note that this is possible since $|L(vw)|=3$. Note that, now, because $\ell(vw) \neq 1$, whatever label we assign to $uv$, we cannot get a conflict between $u$ and $v$. Thus, when labelling $uv$, we just need to make sure that $v$ does not get in conflict with $w$, which can easily be ensured since $|L(uv)|=3$.

We may also assume that $T$ has branching vertices, i.e., vertices with degree at least~$3$. Indeed, if $T$ has no branching vertex, then $T$ is a path, $\Delta(T)=2$, and the claim follows from Theorem~\ref{theorem:cycles}. 
So assume that $T$ has branching vertices.
Root $T$ at any branching vertex $r$. 
This defines the usual root-to-leaf orientation,
through which every non-root vertex has a unique \textit{parent}, i.e., a neighbour that is closer to $r$, and every non-leaf vertex $v$ has \textit{sons}, i.e., neighbours that are farther from $r$, and, more generally, \textit{descendants}, i.e., vertices for which the unique path to $r$ goes through $v$.

Let $u$ be a branching vertex of $T$ that is at farthest distance from $r$. 
Note that we have $u = r$ if $r$ is the unique branching vertex of $T$.
By this choice, $u$ has at least two descendants, all of which have degree at most~$2$. In other words, the descendants of $u$ form $k \geq 2$ disjoint pending paths, none of which has length more than~$2$, as mentioned earlier. There are then $k=p+q \geq 2$ pending paths attached at $u$ formed by its descendants, where $p \geq 0$ of these paths have length~$2$, while $q \geq 0$ of them have length~$1$. We denote by $v_1,\dots,v_p, w_1, \dots, w_q$ the sons of $u$, where $v_1, \dots, v_p$ belong to pending paths of length~$2$, while $w_1,\dots,w_q$ are leaves. We also denote by $v_1', \dots, v_p'$ the neighbour of $v_1,\dots,v_p$, respectively, different from $u$. Thus, the $v_i$'s have degree~$2$, while the $v_i'$'s and the $w_i$'s have degree~$1$. Lastly, we denote by $t$ the parent of $u$, 
if it exists (recall that we have $u=r$ when $T$ has only one branching vertex, in which case $u$ has no parent).

Let $T'=T-\set{v_1,\dots,v_p,v'_1,\dots,v'_p,w_1,\dots,w_q}$. The tree $T'$ is nice, because either $r$ is a branching vertex (case where $u \neq r$) or $T'$ consists in only one vertex (case where $u=r$), and thus $T'$ admits a p-proper $L'$-labelling $\ell'$, where $L'$ denotes the restriction of $L$ to the edges of $T'$. To extend $\ell'$ to a p-proper $L$-labelling of $T$, we just have to assign labels from their lists to the edges incident to the descendants of $u$, so that no  conflict arises. 

We distinguish several cases, based mainly on the value of $q$.

\begin{itemize}
    \item Suppose that $q =0$. Label every edge $uv_i$ with $i \in \{1, \dots, p-1\}$ with an arbitrary label from $L(uv_i)$ different from $1$. Now, label $uv_p$ with a label from $L(uv_p)$ different from~$1$ so that $u$ does not get in conflict with $t$, if it exists (in case it does not, just assign any label different from $1$ to $uv_p$). Note that this is possible since $|L(uv_p)|=3$.
    Lastly, consider every edge $v_iv_i'$. Since $\ell(uv_i) \neq 1$,
    note that $v_i$ and $v_i'$ cannot get in conflict, whatever label from $L(v_iv_i')$ is assigned to $v_iv_i'$. Thus, when labelling $v_iv_i'$, we just need to ensure that $v_i$ and $u$ do not get in conflict, which can be avoided since $|L(v_iv_i')|=3$.
    
    \item Suppose now that $q=1$. Recall that $p \geq 1$ since $k=p+q \geq 2$. We start by labelling, for every $i \in \{1, \dots, p-1\}$, the edge $uv_i$ with any label different from~$1$, chosen from $L(uv_i)$. We then consider $uv_p$, and assign to this edge a label from $L(uv_p)$ different from~$1$ so that the resulting partial product of $u$ is different from~$1$. Note that this is possible since $|L(uv_p)|=3$. Now, note that, by this choice of $\ell(uv_p)$, no matter what $\ell(uw_1)$ is, we cannot get a conflict between $u$ and $w_1$. We then assign as $\ell(uw_1)$ a label from $L(uw_1)$ so that $u$ does not get in conflict with $t$ (if it exists). Lastly, we consider every $i \in \{1, \dots, p\}$, and, to every edge $v_iv_i'$, we assign a value from $L(v_iv_i')$ so that $v_i$ and $u$ do not get in conflict. This results in $\ell$ being p-proper. Recall, in particular, that any two $v_i,v_i'$ cannot be in conflict since $\ell(uv_i)\neq 1$.
\end{itemize}

Suppose now that $q \geq 2$.
We start by stating the following general claim:

\begin{claim}\label{claim:star-cn}
Let $S$ be a star with center $u$ and $q+1 \geq 3$ leaves $t,w_1,\dots,w_q$.
Assume we have a partial labelling $\ell'$ of $S$ where $ut$ is the only edge being assigned a label, $a$, and that $t$ has (virtual) product $\pi_{\ell'}(t)=A$.
If $L$ is a $3$-list assignment to the $uw_i$'s,
then, for every $i \in \{1,\dots,q\}$, we can assign a label from $L(uw_i)$ to $uw_i$, so that $\ell'$ is extended to a labelling $\ell$ of $S$
verifying $\pi_\ell(u) \not \in \{A,\pi_\ell(w_1),\dots,\pi_\ell(w_q)\}$.
\end{claim}

\begin{proofclaim}
Suppose first that $q=2$.
We first assign to $uw_1$ a label from $L(uw_1)$ different from $1/a$.
This way, no matter what label is assigned to $uw_2$, note that $u$ and $w_2$ cannot get in conflict. We now assign a label from $L(uw_2)$ to $uw_2$ so that the resulting product of $u$ is different from $A$ and the product of $w_1$. This is possible since $|L(uw_2)|=3$.

Assume now that $q \geq 3$. We distinguish the following cases:

\begin{itemize}
    \item Assume, w.l.o.g., that the three values in $L(uw_1)$ have pairwise distinct absolute values. To each edge $uw_i$, we associate a variable $x_i$, and we consider the polynomial $$P(x_1,\dots,x_q)= \left(a \prod_{i=1}^q x_i - A \right) \cdot \prod_{i=1}^q \left(a \prod_{j=1}^q x_j-x_i \right).$$
    For every $i \in \{1,\dots,q\}$, we set $y_i=\log x_i$.
    Now the polynomial $P$ becomes equivalent to $$P'(y_1,\dots,y_q)=\left(\log(a)+\sum_{i=1}^q y_i - \log(A)\right) \cdot \prod_{i=1}^q \left(\log(a) + \sum_{j=1}^q y_j - y_i\right).$$ Note that, in the expansion of $P'$, the monomial $y_1^2 y_2 \dots y_q$ has strictly positive coefficient. 
    Thus, by the Combinatorial Nullstellensatz, we can assign values to the $y_i$'s so that $P'$ does not vanish, as long as we are given a set of at least three possible distinct values as $y_1$, and a set of at least two possible distinct values as each of $y_2,\dots,y_q$. In turn, this means we can assign values to the $x_i$'s so that $P$ does not vanish, as long as we have a set of at least three possible values with pairwise distinct absolute values as $x_1$, and a set of at least two possible values with distinct absolute values as each of $x_2,\dots,x_q$. Recall that we made the assumption that the three values in $L(uw_1)$ have pairwise distinct absolute values, while, for every $i \in \{2,\dots,q\}$, there must be at least two values in $L(uw_i)$ with distinct absolute values, since $|L(uw_i)|=3$. Thus, $\ell'$ can correctly be extended to $\ell$, in the desired way.
    
    \item Now assume that every $L(uw_i)$ is of the form $\{\alpha_i,\beta_i,-\beta_i\}$, where $\alpha_i$ and $\beta_i$ are distinct values with the same sign. Let us start from the labelling $\psi$ of $S$ obtained from $\ell'$ after setting $\ell(uw_i)=\alpha_i$ for every $i \in \{1,\dots,q\}$. We denote by $s \in \{-,+\}$ the sign of $\pi_\psi(u)$, while, for every sign $\epsilon \in \{-,+\}$, we denote by $W^\epsilon$ the set of vertices $w_i$ for which the sign of $\pi_\psi(w_i)$ (thus, of $\alpha_i$ and $\beta_i$) is $\epsilon$. Note that $W^-$ and $W^+$ partition the $w_i$'s.
    
    To conclude the proof, we consider two last main cases.
    
    \begin{itemize}
        \item Suppose that $s=+$ and $W^- = \emptyset$.
        We start by assigning label $-\beta_1$ from $L(uw_1)$ to $uw_1$.
        Note that, as long as each $uw_i$ with $i \in \{2,\dots,q\}$ is assigned a label from $\{\alpha_i,\beta_i\}$, we cannot get a conflict between $u$ and $w_i$ due to their products having different signs. Thus, under that convention, the only conflicts we must pay attention to, are along the edges $uw_1$ and, possibly, $ut$ (in case $A$ is negative).
        
        We here assign a variable $x_i$ to each edge $uw_i$ with $i \in \{2, \dots, q\}$, and consider $$P(x_2,\dots,x_q)= \left(-\beta_1 a \prod_{i=2}^q x_i - A\right) \cdot \left(-\beta_1 a \prod_{i=2}^q x_i - \beta_1\right).$$
        For every $i \in \{1,\dots,q\}$, we again set $y_i=\log x_i$.
        Then $P$ is equivalent to $$P'(y_2,\dots,y_q)= \left(\log(-\beta_1a)+ \sum_{i=2}^q y_i - \log(A)\right) \cdot \left(\log(-\beta_1a)+ \sum_{i=2}^q y_i - \log(\beta_1)\right).$$ Recall that $q \geq 3$.
        Then, whatever $q$ is, in the expansion of $P'$ the monomial $y_2y_3$ has strictly positive coefficient. The Combinatorial Nullstellensatz then implies that we can assign values to $y_2,\dots,y_q$ so that $P'$ does not vanish, assuming we have at least two values to choose from for each of $y_2$ and $y_3$, and at least one value to choose from for each of $y_4, \dots, y_q$. From this, we deduce that we can assign values to $vw_2,\dots,vw_q$ from $\{\alpha_2,\beta_2\}, \{\alpha_3,\beta_3\}, \{\alpha_4\}, \{\alpha_5\}, \dots, \{\alpha_q\}$, respectively, so that $u$ is in conflict with none of $w_1$ and $t$. Recall that the resulting sign of $\pi_\ell(u)$ is negative, while the sign of all vertices $w_i$ with $i \in \{2,\dots,q\}$ is positive. Thus, these vertices also cannot be in conflict.
        
        \item Suppose that $s=+$ and $W^- \neq \emptyset$. Assume w.l.o.g. that $w_1 \in W^-$. Recall that, as long as $u$ and $w_1$  get products with different signs by a labelling, they cannot be in conflict. Thus, we here get our conclusion through the Combinatorial Nullstellensatz, by not modelling the possible conflict between $u$ and $w_1$. The precise details are as follows. For every $i \in \{1,\dots,q\}$, let $x_i$ be a variable associated to $uw_i$. We consider the polynomial $$P(x_1,\dots,x_q)= \left(a \prod_{i=1}^q x_i - A \right) \cdot \prod_{i=2}^q \left(a \prod_{j=1}^q x_j-x_i \right).$$
         For every $i \in \{1,\dots,q\}$, we set $y_i=\log x_i$.
         Then $P$ is equivalent to $$P'(y_1,\dots,y_q)=\left(\log(a)+\sum_{i=1}^q y_i - \log(A)\right) \cdot \prod_{i=2}^q \left(\log(a) + \sum_{j=1}^q y_j - y_i\right).$$
         In the expansion of $P'$, the monomial $y_1 \dots y_q$ has strictly positive coefficient, and, thus, by the Combinatorial Nullstellensatz, we can assign labels from $\{\alpha_1,\beta_1\},\dots,\{\alpha_q,\beta_q\}$ to $uw_1,\dots,uw_q$, respectively, resulting in a labelling $\ell$ of $S$ where $u$ gets in conflict with none of $w_2,\dots,w_q,t$. Proceeding that way, recall that the sign of $\pi_\ell(u)$ is positive, while that of $\pi_\ell(w_1)$ is negative. Then, also $u$ and $w_1$ cannot be in conflict, and $\ell$ is p-proper.
    \end{itemize}
\end{itemize}

    To conclude the proof, let us point out that the cases where $s=-$ can be treated in a symmetric way, by considering whether $W^+$ is empty or not.
\end{proofclaim}

We are now ready to conclude the proof of Theorem~\ref{theorem:trees}.
Recall that we have obtained a labelling $\ell'$ of $T'=T-\set{v_1,\dots,v_p,v'_1,\dots,v'_p,w_1,\dots,w_q}$ by induction, and that we are in the case where $u$ is adjacent to $q \geq 2$ leaves (and, possibly, $p$ $v_i$'s and one parent $t$). We start extending $\ell'$ to $T$ by considering every edge $uv_i$ (if such edges exist) and assigning to it a label from $L(uv_i)$ different from~$1$. This is clearly possible, since $|L(uv_i)|=3$. We now apply Claim~\ref{claim:star-cn} to the $uw_i$'s to get all edges incident to $u$ labelled, in such a way that $u$ is not in conflict with any of $t$ (if it exists; if it does not, then note that the claim applies in a very close way) and the $w_i$'s. The main difference here, is that, though we do not have to care about possible conflict between $u$ and the $v_i$'s for now, the claim must be employed with taking into consideration the contribution of the $uv_i$'s to the product of $u$. Lastly, it remains to label every $v_iv_i'$ with a label from $L(v_iv_i')$ so that $v_i$ and $u$ do not get into conflict, which is possible since we have three possible labels. Recall in particular that $v_i$ and $v_i'$ cannot be in conflict since $\ell(uv_i) \neq 1$.
Eventually, $\ell$ is p-proper, as desired.
\end{proof}


\subsubsection*{Planar graphs with large girth}

Recall that a \textit{planar graph} is a graph that can be embedded in the plane so that no two edges cross, and that, for any graph $G$, the \textit{girth} $g(G)$ of $G$ refers to the length of its shortest cycles. In case $G$ has no cycle,  we set $g(G)=\infty$.

Planar graphs with large enough girth are known to be $2$-degenerate and to have low maximum average degree. Thus, the third and fifth items of Corollary~\ref{corollary:existing-bounds} establish~$5$ as a constant upper bound on $\chps(G)$ when $G$ is indeed a nice planar graph with large girth. In what follows, we improve this upper bound down to~$4$ when $g(G) \geq 16$, getting closer to the List Multiplicative 1-2-3 Conjecture for this class of graphs. Our proof involves arguments that are reminiscent to those used to prove Theorem~\ref{theorem:trees}, combined together with the following structural result:

\begin{theorem}[e.g. Ne\v{s}et\v{r}il, Raspaud, Sopena~\cite{NRS97}]\label{theorem:threads}
If $G$ is a planar graph with girth $g(G) \geq 5\ell +1$ for some $\ell \geq 1$, then either:
\begin{itemize}
    \item $\delta(G)=1$, or
    \item $G$ contains an \textit{$\ell$-thread}, i.e., a path $uv_1\dots v_\ell w$ where $d(u),d(w) \geq 2$, and $d(v_i)=2$ for every $i \in \{1,\dots,\ell\}$.
\end{itemize}
\end{theorem}

We are now ready to prove our result.

\begin{theorem}\label{theorem:planar-girth16}
If $G$ is a nice planar graph with girth $g(G) \geq 16$, then $\chps(G) \leq 4$.
\end{theorem}

\begin{proof}
Assume the claim is wrong, and let $G$ be a minimal counterexample to the claim. We may assume that $G$ is connected, and, due to Theorems~\ref{theorem:cycles} and~\ref{theorem:trees}, that $\Delta(G) \geq 3$ and that $G$ is not a tree. Let $L$ be a $4$-list assignment to the edges of $G$. We prove the result by contradicting the existence of $G$, i.e., by showing that $G$ admits p-proper $L$-labellings, whatever $L$ is.

If $\delta(G) \geq 2$, then, by Theorem~\ref{theorem:threads}, we can find a $3$-thread $uv_1v_2v_3w$ in $G$. In that case, we consider $G'=G-v_2$. Note that $G'$ may consist in up to two connected components, each of which has at least two edges (since $d(u),d(w) \geq 2$, by the assumption on $\delta(G)$) and girth at least~$16$ (in case there is only one connected component, $G'$ might be a tree; in that case, $g(G')=\infty$, and the girth condition remains true). So $G'$ is nice and planar, and, by minimality of $G$, there is a p-proper $L'$-labelling $\ell'$ of $G'$, where $L'$ denotes the restriction of $L$ to the edges of $G'$. To obtain a contradiction, it now suffices to extend $\ell'$ to a p-proper $L$-labelling of $G$, and, for this, we just have to assign labels from $L(v_1v_2)$ and $L(v_2v_3)$ to $v_1v_2$ and $v_2v_3$, respectively, so that no conflict arises. This can clearly be done since $|L(v_1v_2)|=|L(v_2v_3)| = 4$, by first assigning to $v_1v_2$ a label  different from $\ell'(v_3w)$ for which $v_1$ and $u$ get different partial products, and then assigning to $v_2v_3$ a label so that $v_1$ and $v_2$ are not in conflict, and similarly for $v_3$ and $w$.

We may thus assume that $\delta(G)=1$. Since $G$ is not a tree, this means that, by repeatedly removing vertices of degree~$1$ while there are some, we end up with a planar connected graph $G^-$ such that $\delta(G^-) \geq 2$ and $g(G^-) \geq 16$. More precisely, for every $v \in V(G) \cap V(G^-)$, we can denote by $T_v$ the pending tree rooted at $v$ in $G$, which, if $d_G(v)=d_{G^-}(v)$, is reduced to the single vertex $v$. Then $G^-$ is obtained from $G$ by contracting every $T_v$ to $v$. For every $v \in V(G) \cap V(G^-)$, we deal, in $G$, with $T_v$ through the terminology introduced in the proof of Theorem~\ref{theorem:trees} (in particular, the notions of parent, son, descendant and branching vertex have the exact same meaning).

Because $g(G^-) \geq 16$, then, by Theorem~\ref{theorem:threads}, we deduce that $G^-$ has a $3$-thread $P=uv_1v_2v_3w$. Note that $P$ also exists back in $G$, the difference being that $v_1,v_2,v_3$ might each be the root of a pending tree (denoted $T_{v_1},T_{v_2},T_{v_3}$, respectively, following our terminology) that might have edges. In case we have $V(T_{v_i})=\{v_i\}$ for every $i \in \{1,2,3\}$, then note that $P$ is actually a $3$-thread in $G$, in which case a contradiction can be obtained in the similar way as in the previous case $\delta(G) \geq 2$. Thus, in what follows, we assume that some of $T_{v_1},T_{v_2},T_{v_3}$ are not reduced to a single vertex.

By arguments similar to some used in the proof of Theorem~\ref{theorem:trees}, we may assume that none of $T_{v_1},T_{v_2},T_{v_3}$ has 1) a non-root branching vertex, or 2) a pending path of length at least~$3$ (remind, in particular, that in the current context there is even more room for labelling extensions, due to $L$ being a $4$-list assignment). This means that each $T_{v_i}$ is a subdivided star with center $v_i$, where the pending paths attached to $v_i$ (if any) have length~$1$ or $2$.

We start by handling a very particular case, which is when every $T_{v_i}$ has only one edge $v_iv_i'$, i.e., is a star with a single edge $v_iv_i'$. In this case, we consider $G'=G-v_2$. A p-proper $L'$-labelling of $G'$ (where, again, $L'$ denotes the restriction of $L$ to $G'$), which exists by minimality of $G$, can then be extended to a p-proper $L$-labelling of $G$, a contradiction, by first labelling $v_1v_2$ with a label from $L(v_1v_2)$ so that no conflict between $v_1$ and its two neighbours different from $v_2$ arises, then labelling $v_2v_3$ with a label from $L(v_2v_3)$ so that 1) no conflict between $v_3$ and its two neighbours different from $v_2$ arises, and 2) $v_2$ gets partial product different from~$1$; and lastly labelling the edge $v_2v_2'$ of $T_{v_2}$ with a label from $L(v_2v_2')$ so that no conflict between $v_2$ and its two neighbours different from $v_2'$ arises. Recall, in particular, that $v_2$ and $v_2'$ cannot be in conflict due to how $v_2v_3$ was labelled.
Note also that lists of four labels are indeed sufficient to achieve this whole process.

In the more general case, let us consider the graph $G'=G-(V(T_{v_1}) \setminus \{v_1\})-V(T_{v_2})-(V(T_{v_3})\setminus\{v_3\})$ (obtained by removing the non-root vertices of $T_{v_1}$ and $T_{v_3}$, and the whole of $T_{v_2}$). By arguments used earlier in the case where $\delta(G) \geq 2$, there is a p-proper $L'$-labelling $\ell'$ of $G'$, where $L'$ denotes the restriction of $L$ to the edges of $G'$. Our goal, to get a final contradiction, is to extend $\ell'$ in a p-proper way to the edges $v_1v_2$, $v_2v_3$ and those in $T_{v_1},T_{v_2},T_{v_3}$, assigning labels from their respective lists, so that a p-proper $L$-labelling of $G$ results.

We start by assigning labels from $L(v_1v_2)$ and $L(v_2v_3)$ to $v_1v_2$ and $v_2v_3$, respectively, in such a way that, for the resulting partial products of $v_1,v_2,v_3$, 1) $v_2$ is in conflict with none of $v_1$ and $v_3$, 2) $v_1$ is not in conflict with $u$, 3) $v_3$ is not in conflict with $w$, and 4) none $v_i$ of $v_1,v_2,v_3$ for which $T_{v_i}$ contains only one edge, gets product~$1$ as a result. This is possible to achieve since $|L(v_1v_2)|=|L(v_2v_3)|=4$. More precisely, this can be achieved by labelling $v_1v_2$ first and $v_2v_3$ second if $T_{v_1}$ has only one edge, or by labelling $v_2v_3$ first and $v_1v_2$ second otherwise. Recall, in particular, that we have treated separately the case where all of $T_{v_1},T_{v_2},T_{v_3}$ have only one edge, so we are not in that case; the fourth condition must thus be fulfilled for at most two of the $v_i$'s.

It now remains to label the edges from the $T_{v_i}$'s. We achieve this by considering $T_{v_1}$, $T_{v_2}$ and $T_{v_3}$ in turn, so that, once every $T_{v_i}$ has been treated, no vertex in $V(T_{v_1}) \cup \dots \cup V(T_{v_i})$ is involved in conflicts, and none of the vertices in $V(T_{v_{i+1}}) \cup \dots \cup V(T_{v_3})$ had its product altered. This way, the desired p-proper $L$-labelling of $G$ will result once $T_{v_3}$ has been treated. In what follows, we focus on $T_{v_1}$, but the arguments apply similarly for $T_{v_2}$ and $T_{v_3}$.

Recall that $T_{v_1}$ consists of some (possibly none) pending paths of length~$1$ or~$2$ attached to $v_1$. Let us assume that $p \geq 0$ of these paths have length~$2$, while $q \geq 0$ of them have length~$1$. We denote by $b_1,\dots,b_p$ the sons of $v_1$ that belong to the pending paths of length~$2$, while we denote by $c_1,\dots,c_q$ those from the pending paths of length~$1$. Finally, for every $i \in \{1,\dots,p\}$, we denote by $b_i'$ the son of $b_i$ in $T_{v_1}$. By how $v_1v_2$ was labelled earlier, note that we already have the desired conclusion around $v_1$ if $p=q=0$. We thus focus on the cases where $p+q>0$.

\begin{itemize}
    \item The cases where $q \in \{0,1\}$ can be treated quite similarly as the cases $q=0$ and $q=1$ in the proof of Theorem~\ref{theorem:trees}. Namely, we first label the edges $v_1b_1,\dots,v_1b_{p-1}$ (if such exist) with labels different from~$1$ from their respective lists. If $q=0$, then we label $v_1b_p$ with a label different from~$1$ from its list, with making sure that the resulting product of $v_1$ is different from that of $u$ and $v_2$. Otherwise, if $q=1$, then we label $v_1b_p$ with a label different from~$1$ from its list, with making sure that the resulting partial product of $v_1$ does not get equal to~$1$ (if $p=0$, then recall that this property is already verified at $v_1$, due to how $v_1v_2$ and $v_2v_3$ have been labelled). Still in the case where $q=1$, this guarantees that $v_1$ and $c_1$ cannot get in conflict no matter how $v_1c_1$ is labelled; thus, we can label $v_1c_1$ with a label from its list so that $v_1$ does not in conflict with $u$ and $v_2$. Note that lists of size~$4$ are sufficient to achieve these conditions in all cases. We lastly label every edge $b_ib_i'$ (if any) with a label from its list, with making sure that $b_i$ does not get in conflict with $v_1$. Because $v_1b_i$ was assigned a label different from~$1$, recall that $b_i$ and $b_i'$ cannot be in conflict.
    
    \item The cases where $q=2$ can be treated quite similarly. Start by labelling every edge $v_1b_i$ (if there are some) with a label different from~$1$ from its list. Then, label $v_1c_1$ with a label from its list, so that the resulting partial product of $v_1$ does not get equal to~$1$. Last, label $v_1c_2$ with a label from its list, so that $v_1$ gets in conflict with none of $u$, $v_2$ and $c_1$. Note that this is possible, since we do not have to care about a possible conflict between $v_1$ and $c_2$, and $|L(v_1c_2)|=4$. To conclude, we can eventually label the $b_ib_i'$'s just as in the previous case.
\end{itemize}
    
The general case is when $q \geq 3$. We need a generalisation of Claim~\ref{claim:star-cn} to the current context.

\begin{claim}\label{claim:star-planar}
Let $S$ be a star with center $u$ and $q+2 \geq 5$ leaves $t,t',w_1,\dots,w_q$. Assume we have a partial labelling $\ell'$ of $S$ where $ut$ and $ut'$ are the only edges being assigned a label, $a$ and $a'$, respectively, and that $t$ and $t'$ have (virtual) product $\pi_{\ell'}(t)=A$ and $\pi_{\ell'}(t')=A'$. If $L$ is a $4$-list assignment to the $uw_i$'s, then, for every $i \in \{1,\dots,q\}$, we can assign a label from $L(uw_i)$ to $uw_i$, so that $\ell'$ is extended to a labelling $\ell$ of $S$ verifying $\pi_\ell(u) \not \in \{A, A', \pi_\ell(w_1),\dots,\pi_\ell(w_q)\}$.
\end{claim}

\begin{proofclaim}
Note that each $L(uw_i)$ contains two, three or four values with pairwise distinct absolute values.
We consider several cases based on that fact.

\begin{itemize}
    \item Assume, w.l.o.g., that the four values in $L(uw_1)$ have pairwise distinct absolute values. To each edge $uw_i$, we associate a variable $x_i$, and we consider the polynomial
    $$P(x_1,\dots,x_q)= \left(aa' \prod_{i=1}^q x_i - A \right) \cdot \left(aa' \prod_{i=1}^q x_i - A'\right) \cdot \prod_{i=1}^q \left(aa' \prod_{j=1}^q x_j-x_i \right).$$
    For every $i \in \{1,\dots,q\}$, we set $y_i=\log x_i$.
    Then $P$ gets equivalent to
    \begin{equation*}
    \begin{split}
        P'(y_1,\dots,y_q)=&\left(\log(aa')+\sum_{i=1}^q y_i - \log(A)\right) \cdot \left(\log(aa')+\sum_{i=1}^q y_i - \log(A')\right) \\ & \cdot 
        \prod_{i=1}^q \left(\log(aa') + \sum_{j=1}^q y_j - y_i\right).
    \end{split}
    \end{equation*}
    In the expansion of $P'$, the monomial $y_1^3 y_2 \dots y_q$ has strictly positive coefficient. 
    Thus, by the Combinatorial Nullstellensatz, we can assign values to the $y_i$'s so that $P'$ does not vanish, as long as we are given a set of at least four possible distinct values as $y_1$, and a set of at least two possible distinct values as each of $y_2,\dots,y_q$. Regarding $P$, this implies we can assign values to the $x_i$'s so that $P$ does not vanish, assuming we have a set of a least four possible values with pairwise distinct absolute values as $x_1$, and a set of at least two possible values with distinct absolute values as each of $x_1,\dots,x_q$. This is met in the current case, since $L(uw_1)$ is assumed to have four values with pairwise distinct absolute values, and $|L(uw_i)|=4$ for every $i \in \{2,\dots,q\}$. Thus, $\ell'$ can be extended to $\ell$ as desired.
    
    \item Assume now that, w.l.o.g., both $L(uw_1)$ and $L(uw_2)$ include three values with pairwise distinct absolute values. Then the same conclusion as in the previous case can be reached from considering the monomial $y_1^2y_2^2y_3\dots y_q$ in the expansion of $P'$.
    
    \item We can thus assume that none of the two previous cases applies, i.e., that, w.l.o.g., $L(uw_1)$ includes two or three values with pairwise distinct absolute values, while $L(uw_2),\dots,L(uw_q)$ include each exactly two values with pairwise distinct absolute values. In other words, we have $L(uw_i)=\{\alpha_i,-\alpha_i,\beta_i,-\beta_i\}$ for every $i \in \{2,\dots,q\}$, for some distinct $\alpha_i,\beta_i$, while $L(uw_1)=\{\alpha_1,-\alpha_1,\beta_1,-\beta_1\}$ or $L(uw_1)=\{\alpha_1,-\alpha_1,\beta_1,\gamma_1\}$, for some distinct $\alpha_1,\beta_1,\gamma_1$. To conclude the proof, we consider a few more cases:
    
    \begin{itemize}
        \item Assume first that $A$ and $A'$ have the same sign $s \in \{-,+\}$. For every $i \in \{1,\dots,q-2\}$, let us assign to $uw_i$ a label with sign $s$ from its list. Then:
        
        \begin{itemize}
            
            
            
            
            \item If $s$ and the sign of the partial product of $u$ are the same, then we assign to $uw_{q-1}$ a label with sign $s$ from its list, chosen so that the partial product of $u$ gets different from~$1$. Note that this is possible, since $L(uw_{q-1})$ contains two  values with sign $s$. This guarantees that $u$ and $w_q$ cannot be in conflict, whatever the label of $uw_q$ is. We then assign to $uw_{q}$ a label with sign $-s$ from its list, so that all edges are labelled and no conflict remains. In particular, $u$ gets product with sign $-s$, while only $w_q$ has this property.
            
            \item Otherwise, i.e., if $s$ and the sign of the partial product of $u$ are different, then we assign to $uw_{q-1}$ and $uw_q$ a label with sign $s$ from their lists. As a result, no conflict remains, since $u$ is the only vertex with product being of sign $-s$.
        \end{itemize}
        
        \item Now assume that $A$ and $A'$ have different signs, say $A$ is positive while $A'$ is negative. We here start by assigning, for every $i \in \{1,\dots,q-2\}$, a positive label to $uw_i$ from its list $L(uw_i)$. Now:
        
        \begin{itemize}
            \item If currently $u$ has negative product, then we assign to $uw_{q-1}$ and $uw_q$ a positive label from their respective lists, with making sure that the product of $u$ gets different from $A'$. This is possible since $L(uw_{q-1})$ and $L(uw_q)$ have two positive values each. Since only $u$ and $t'$ have negative product, no conflict remains.
            
            \item Otherwise, i.e., $u$ currently has positive product, then we first assign a positive label to $uw_{q-1}$ from its list, chosen so that the current product of $u$ does not get equal to~$1$. This is possible, since $L(uw_{q-1})$ contains two positive values. This guarantees that $u$ and $w_q$ cannot get in conflict. We then assign to $uw_q$ a negative label from $L(uw_q)$, chosen so that $u$ gets product different from~$A'$. This is possible since $L(uw_q)$ contains two negative values. Since only $A'$ and the products of $u$ and $w_q$ are negative, no conflict remains.
        \end{itemize}
    \end{itemize}
\end{itemize}

In all cases, we end up with the desired labelling $\ell$, which concludes the proof.
\end{proofclaim}

We can now conclude the case $q \geq 3$ of the proof of Theorem~\ref{theorem:planar-girth16}, thus proving the whole statement. We start by labelling every edge $v_1b_i$ (if any) with any label different from~$1$ from its list $L(v_1b_i)$. We now apply Claim~\ref{claim:star-planar} to get all $v_1c_i$'s labelled with labels from their lists, so that $v_1$ is not in conflict with any of $u$, $v_2$ and the $c_i$'s. This can be done by applying Claim~\ref{claim:star-planar} with $v_1$, $u$ and $v_2$ playing the role of $u$, $t$ and $t'$, respectively, $\pi_{\ell'}(u)$ and $\pi_{\ell'}(v_2)$ playing the role of $A$ and $A'$, respectively, $\ell'(uv_1)\prod_{i=1}^p \ell(ub_i)$ and $\ell'(v_1v_2)$ playing the role of $a$ and $a'$, respectively, and the $c_i$'s playing the role of the $w_i$'s. It remains to label the $b_ib_i'$'s (if any), and, for each such edge $b_ib_i'$, it suffices to assign a label from its list so that $b_i$ and $v_1$ do not get in conflict. Recall that we do not have to mind about a possible conflict between $b_i$ and $b_i'$, since $\ell(v_1b_i) \neq 1$.
\end{proof}

\subsubsection{Subcubic graphs}

We now consider subcubic graphs, i.e., graphs with maximum degree~$3$. Note that, at this point, the best upper bound we have on $\chps$ for these graphs is~$5$, obtained from Theorem~\ref{theorem:upper-bound-cn}. We get one step closer to the List Multiplicative 1-2-3 Conjecture for this class of graphs, by lowering the upper bound down to~$4$ in the next result.

\begin{theorem}\label{theorem:subcubic}
If $G$ is a nice subcubic graph, then $\chps(G) \leq 4$.
\end{theorem}

\begin{proof}
Assume the claim is wrong, and consider $G$ a minimal counterexample to the claim.
Clearly, $G$ is connected.
Let $L$ be a $4$-list assignment to the edges of $G$.
We prove below that $G$ admits a p-proper $L$-labelling whatever $L$ is, a contradiction.
To that aim, we first show that $G$ is cubic:

\begin{itemize}
    \item Assume first that $\delta(G)=1$, and consider $u$ a degree-$1$ vertex of $G$ with unique neighbour $v$. 
    
    \begin{itemize}
        \item Assume first that $d(v)=2$, and let $w$ denote the second neighbour of $v$. Set $G'=G-\{u,v\}$. We can assume that $G'$ is nice, as otherwise $G$ would be the path of length~$3$, in which case even $\chps(G) \leq 3$ holds by Theorem~\ref{theorem:cycles}, a contradiction. Then, by minimality of $G$, there is a p-proper $L'$-labelling $\ell'$ of $G'$, where $L'$ denotes the restriction of $L$ to the edges of $G'$. We extend $\ell'$ to a p-proper $L$-labelling of $G$, getting a contradiction, by correctly assigning labels to $uv$ and $vw$ from their respective lists. We first label $vw$, by assigning a label from $L(vw)$ that is different from~$1$, and so that $w$ does not get in conflict with any of its at most two other neighbours different from $v$. Note that this is possible since $|L(vw)|=4$. We can now extend the labelling to $uv$ by assigning a label from $L(uv)$ so that $v$ does not get in  conflict with $w$. Note that by how $vw$ was labelled, $u$ and $v$ cannot get in conflict. 
        
        \item Assume now that $d(v)=3$, and let $w_1,w_2$ denote the two neighbours of $v$ different from $u$. Set $G'=G -\{u,v\}$. We can assume that $G'$ is nice, as otherwise either 1) one of the $w_i$'s is a degree-$2$ vertex adjacent to a $1$-vertex, or 2) $w_1w_2$ exists and both $w_1$ and $w_2$ have degree~$2$. In the former case, we fall into the previous case (where $d(v)=2$) we have handled. In the latter case, $G$ has only four edges and the claim can be checked by hand. So $G'$ is nice, and, by minimality of $G$, there is a p-proper $L'$-labelling $\ell'$ of $G'$, where $L'$ denotes the restriction of $L$ to the edges of $G'$. To extend it to one of $G$, thus getting a contradiction, we proceed as follows. For every $i \in \{1,2\}$, note that there are at least two values $a_i,b_i \in L(uw_i)$ that can be assigned to $vw_i$ without causing any  conflict between $w_i$ and its at most two neighbours different from $v$. We assign labels to $vw_1$ and $vw_2$ from $\{a_1,b_1\}$ and $\{a_2,b_2\}$, respectively, so that the product of these two labels is different from~$1$. It then suffices to assign to $uv$ a label from~$L(uv)$ so that $v$ gets in  conflict with none of $w_1$ and $w_2$, which is possible since $|L(uv)|=4$. Again, $u$ and $v$ cannot be in  conflict due to how $vw_1$ and $vw_2$ have been labelled.
    \end{itemize}
    
    \item Assume now that $\delta(G)=2$, and consider $u$ a degree-$2$ vertex of $G$ with neighbours $v_1,v_2$. By the minimum degree assumption, each of $v_1$ and $v_2$ has one or two neighbours different from $u$. We here consider $G'=G-u$. We can assume that $G'$ is nice, as, because $\delta(G)=2$, otherwise it would mean that $v_1v_2$ is the only other edge, thus that $G$ is $C_3$, the cycle of length~$3$, in which case $\chps(G) \leq 3$ holds by Theorem~\ref{theorem:cycles}, a contradiction. So $G'$ admits a p-proper $L'$-labelling $\ell'$, where $L'$ is the restriction of $L$ to the edges of $G'$. We show that this p-proper labelling can be extended to $uv_1$ and $uv_2$ by assigning labels from their lists, thereby getting a contradiction. 
    
    Let $x_1,x_2$ be variables associated to $uv_1$ and $uv_2$, respectively. Let us denote by $y_1,y_2$ the values $\pi_{\ell'}(v_1),\pi_{\ell'}(v_2)$, respectively. Let us now consider the  polynomial $$P(x_1,x_2)=(x_1x_2-x_1y_1) \cdot (x_1x_2-x_2y_2) \cdot 
    \prod_{w \in N(v_1) \setminus \{u\}}(x_1y_1-\pi_{\ell'}(w)) \cdot 
    \prod_{w \in N(v_2) \setminus \{u\}}(x_2y_2-\pi_{\ell'}(w)).$$ If $x_1$ and $x_2$ can be assigned values in $L(uv_1)$ and $L(uv_2)$, respectively, so that $P$ does not vanish, then we get a p-proper $L$-labelling of $G$. Since $x_1$ and $x_2$ are the only variables of $P$, it is easy to see that, in the expansion of $P$, the monomial $M$ with largest degree is either $x_1^4x_2^4$ (when $d(v_1)=d(v_2)=3$), $x_1^3x_2^4$ (when $d(v_1)=2$ and $d(v_2)=3$), $x_1^4x_2^3$ (when $d(v_1)=3$ and $d(v_2)=2$) or $x_1^3x_2^3$ (when $d(v_1)=d(v_2)=2$). In all case, since $M$ has nonzero coefficient, then, by the Combinatorial Nullstellensatz, desired values for $x_1$ and $x_2$ can be chosen from lists of size at least~$5$, thus from lists of size at least~$4$ if we are guaranteed that they do not include~$0$ (due to the first two factors of $P$). From this, we deduce that a p-proper $L$-labelling of $G$ can be obtained from $\ell'$, a contradiction.
\end{itemize}

Thus, from now on, $G$ can be assumed to be cubic.
Let $C = u_1\dots u_pu_1$ be a smallest induced cycle of $G$. For every $i \in \{1,\dots,p\}$, we denote by $u'_i$ the neighbour of $u_i$ which does not belong to $C$. 
Let $G' = G - E(C)$.
Note that $G'$ is nice, since the $u_i$'s have degree~$1$ and are not adjacent in $G'$, while all other vertices have degree~$3$.
Thus, by minimality of $G$, there is a p-proper $L'$-labelling $\ell'$ of $G'$, where $L'$ denotes the restriction of $L$ to the edges of $G'$. Our goal is to extend it to the edges of $C$ in a p-proper way to an $L$-labelling of $G$, thereby getting a final contradiction.
 
To ease the exposition of the upcoming arguments, let us introduce some notation. For every $i \in \{1,\dots,p\}$, we set $L_i = L(u_iu_{i+1})$,  $a'_i = \ell'(u_iu'_i)$ and $A'_i = \frac{\pi_{\ell'}(u'_i)}{a'_i}$ (where, here and further, we set $u_{p+1}=u_1$ and $u_0=u_p$).
For some set $X$ of values and $\lambda \in \mathbb{R}^*$, we define $\lambda X = \set{\lambda x : x \in X}$ and $\frac{\lambda}{X} = \set{\frac{\lambda}{x} : x \in X}$. For two sets $X$ and $Y$, we define $X Y = \set{xy : x \in X, y \in Y}$.

The proof goes by distinguishing several cases depending on some lists by $L$ and on the structure of $G$. In each considered case, it is implicitly assumed that none of the previous cases applies.

\begin{enumerate}
    \item \emph{There are $i_0 \in \{1,\dots,p\}$ and $\alpha \in L_{i_0-1}$  such that, for all $\alpha' \in L_{i_0}$, we have $\alpha\alpha' \neq A'_{i_0}$.} 
    
    W.l.o.g., assume that $i_0 =1$. The assumption implies that $u_1$ and $u'_1$ can never be in conflict in an extension of $\ell'$ assigning label $\alpha$ to $u_{p}u_1$. Let us thus start by assigning label $\alpha$ to $u_pu_1$.
    We then consider the other edges $u_{p-1}u_p, u_{p-2}u_{p-1},\dots,u_1u_2$ of $C$ one by one, following this exact ordering. For every edge $u_iu_{i+1}$ considered that way, we assign a label from $L(u_iu_{i+1})$ chosen in the following manner:
    
    \begin{itemize}
        \item If $i \in \{3,\dots,p-1\}$, then we assign to $u_iu_{i+1}$ a label so that $u_{i+1}$ is in conflict with neither ${u_{i+2}}$ nor $u_{i+1}'$. Note that this is possible since $|L(u_iu_{i+1})|=4$. In the case where $i=p-1$, we note that $u_{i+2}=u_1$ is a vertex whose product is not fully determined yet; the conflict between $u_p$ and $u_1$ will actually be taken care of in a later stage of the extension process.
        
        \item If $i=2$, then we assign to $u_2u_3$ a label so that $u_3$ is in conflict with neither $u_4$ not $u_3'$, and the resulting partial product of $u_2$ gets different from the partial product of $u_1$.  This is possible, since $|L(u_2u_3)|=4$. In case $p=3$ and, thus, $u_4=u_1$, the possible conflict between $u_3$ and $u_1$ will be handled during the next step of the process.
        
        \item If $i=1$, the we assign to $u_1u_2$ a label so that $u_2$ gets in conflict with neither $u_3$ not $u_2'$, and $u_1$ and $u_p$ are not in conflict. Again, this is possible because $|L(u_1u_2)|=4$. Recall further that $u_1$ and $u_2$ cannot be in conflict due to the choice of the label assigned to $u_2u_3$. Also, $u_1$ and $u_1'$ cannot be in conflict by the initial assumption on $\alpha$.
    \end{itemize}
    
    Thus, once the whole process has been carried out, we get an $L$-labelling of $G$ which is p-proper, a contradiction.
    \label{subcubic:item:1}
\end{enumerate}

Since Case~1 does not apply, then, throughout what follows, for every $i \in \set{1,\dots,p}$, we have
\begin{equation}
    L_{i-1} = \frac{A'_i}{L_{i}} \text{ and } L_{i} = \frac{A'_i}{L_{i-1}}. \label{eq:subcubic-1}
\end{equation}
\begin{enumerate}[resume]
    \item \emph{There are $i_0 \in \{1,\dots,p\}$ and $\alpha \in L_{i_0}$ such that, for all $\alpha' \in L_{i_0+2}$, we have $\alpha  a'_{i_0+1} \neq \alpha' a'_{i_0+2}$.} 
    
    W.l.o.g., assume that $i_0 =1$. The assumption implies that $u_2$ and $u_3$ can never be in conflict in an extension of $\ell'$ assigning label $\alpha$ to $u_1u_2$. Let us thus assign label $\alpha$ to $u_1u_2$.
    We then consider the other edges of $C$, and label them with labels from their respective lists so that no conflict arises.
    We consider a special value of $p$, before considering the general case.
    
    \begin{itemize}
        \item Assume first that $p=3$, i.e., $C$ is a triangle. We start by assigning a label from $L(u_2u_3)$ to $u_2u_3$ so that $u_2$ does not get in conflict with $u_2'$, and the partial product of $u_3$ gets different from the partial product of $u_1$. Note that this is possible since $|L(u_2u_3)|=4$. We then assign a label from $L(u_1u_3)$ to $u_1u_3$ so that $u_1$ gets in conflict with neither $u_1'$ nor $u_2$, and $u_3$ does not get in conflict with $u_3'$. Again, such a label exists since $|L(u_1u_3)|=4$. Recall that $u_1$ and $u_3$ cannot be in conflict due to how $u_2u_3$ was labelled. Also, $u_2$ and $u_3$ cannot be in conflict by the assumption on $\alpha$.
        
        \item Otherwise, i.e., $p \geq 4$, we start by assigning a label from $L(u_2u_3)$ to $u_2u_3$ so that $u_2$ and $u_2'$ do not get in conflict. We then consider the remaining edges $u_pu_1,u_{p-1}u_p,\dots,u_3u_4$ of $C$ one by one, following this exact ordering. For every edge $u_iu_{i+1}$ considered that way, we assign a label from $L(u_iu_{i+1})$ chosen in the following way:
        
        \begin{itemize}
            \item If $i \in \{5,\dots,p\}$, then we assign to $u_iu_{i+1}$ a label chosen so that $u_{i+1}$ gets in conflict with neither $u_{i+2}$ nor $u_{i+1}'$. This is possible since $|L(u_iu_{i+1})|=4$.
            
            \item If $i=4$, then we assign to $u_4u_5$ a label chosen so that $u_5$ gets in conflict with neither $u_6$ nor $u_5'$, and the partial product of $u_4$ does not get equal to the partial product of $u_3$. This is possible since $|L(u_4u_5)|=4$.
            
            \item If $i = 3$, then we assign to $u_3u_4$ a label so that $u_4$ gets in conflict with neither $u_5$ nor $u_4'$, and $u_3$ does not get in conflict with $u_3'$. Again, this is possible since $|L(u_3u_4)|=4$. Recall that $u_4$ and $u_3$ cannot be in conflict due to how $u_4u_5$ has been labelled. Also, $u_2$ and $u_3$ cannot be in conflict by the assumption on $\alpha$.
        \end{itemize}
    \end{itemize}
    
    Thus, in all cases, we get a p-proper $L$-labelling of $G$, a contradiction.
\end{enumerate}

Since Case~2 does not apply in what follows, then, for every $i \in \set{1,\dots,p}$, we have
\begin{equation}
    L_i = \frac{a'_{i+2}}{a'_{i+1}} L_{i+2}. \label{eq:subcubic-2}
\end{equation}

\begin{enumerate}[resume]
    \item \emph{$G$ is $K_4$, the complete graph on four vertices.} 
    
    Here, $C$ is a cycle $u_1u_2u_3u_1$ of length~$3$, and we have $u' = u'_1 = u'_2 = u'_3$. Also, $\ell'$ assigns labels to the three edges incident to $u'$, since $G'$ is a star. 
    Note that, as long as we label the edges of $C$ last and handle all conflicts at that point, then, prior to  labelling $C$, we might actually change the labels assigned to $u_1u',u_2u',u_3u'$ by $\ell'$ for other labels from their respective lists.
    
    Note now that, for any choice of label $a_3'$ from  $L(u_3u')$ assigned to $u_3u'$, Identity~(\ref{eq:subcubic-2}) must apply, i.e., we must have $L_1 = \frac{a'_{3}}{a'_{2}} L_{3}$, as otherwise previous Case~2 would apply the very same way. This implies that $|L_3| \geq 5$, a contradiction, by the following arguments. Since $|L(u_3u')| = 4$, there are at least two values $x,y \in L(u_3u')$ with distinct absolute values, say $\abs{x}<\abs{y}$. Start by assigning label $x$ to $L(u_3u')$; because Identity~(\ref{eq:subcubic-2}) applies, we deduce that for every $\alpha \in L_1$ we have $x\alpha \in L_3$. The other way around, we have $L_3=L_3'=\{x\alpha : \alpha \in L_1\}$ and $|L_3'|=|L_3|=4$. Now change the label of $u_3u'$ to $y$. Because $\abs{x}<\abs{y}$, we deduce that, for an $\alpha \in L_1$ with largest absolute value, $y\alpha \not \in L_3'$. This implies that $L_3$ must contain a fifth value not in $L_3'$ for Identity~(\ref{eq:subcubic-2}) to apply with $y$.
    
    \item \emph{$p=3$ and $C$ shares an edge with another triangle.} 
    
    Assume $u_1u_2$ belongs to a triangle $u'u_1u_2u'$ different from $C$, where $u' = u'_1 = u'_2$ is the common neighbour of $u_1$ and $u_2$ different from $u_3$.
    Because we are not in Case~3, we have $u_3'\neq u'$, and $u'$ has a neighbour $w \not \in V(C)$.
    Note that, by $\ell'$, there are actually three possible values in $L(u_2'u')$ that can be assigned to $u_2'u'$ without causing $u'$ to be in conflict with $w$, thus two such values $x,y$, with, say, $|x|<|y|$. Start by setting $a_2'=y$. By an application of Identity~(\ref{eq:subcubic-2}) (which applies as otherwise Case 2 would), we deduce that $L_1 = \frac{a'_{3}}{a'_{2}} L_{3}$, which reveals the exact four values in $L_3$. Now, just as in previous Case~3, we note that by changing the value of $a_2'$ to $x$ and applying Identity~(\ref{eq:subcubic-2}) again, we deduce that $L_3$ must contain a fifth value not among the previous four revealed ones. This is a contradiction.
\end{enumerate}

At this point, note that if we modify the label $a'_i$ assigned to any edge $u_iu_i'$ by $\ell'$, then this has no impact on the value $A_{i+1}'$ (and, symmetrically, on $A_{i-1}'$). Indeed, if modifying $a_i'$ also modified $A_{i+1}'$, then this would imply that $u_iu_i'$ is incident to $u_{i+1}'$, thus that $u'_i=u'_{i+1}$. But, in this case, we would deduce that $u_iu_{i+1}u'_iu_i$ is a triangle sharing an edge with $C$, thereby getting a contradiction to the fact that none of Cases~3 and~4 applies.

By manipulating Identities~(\ref{eq:subcubic-1}) and~(\ref{eq:subcubic-2}), note that we can establish the relationship
\begin{equation}
    L_i = \frac{a'_{i+2}A'_{i+2}}{a'_{i+1}A'_{i+1}}  L_{i} = \frac{a'_{i+1}A'_{i+1}}{a'_{i+2}A'_{i+2}}  L_{i}\label{eq:subcubic-3}
\end{equation}
between any list $L_i$ and some of the $a_i'$'s and $A_i'$'s.
For every $i \in \{1,\dots,p\}$,
we define $\lambda_i = \frac{A'_{i+1}}{a'_{i+2}A'_{i+2}}$;
then, $L_i=a_{i+1}'\lambda_iL_i$ by the above.

\begin{enumerate}[resume]
\item \emph{There are $i \in \{1,\dots,p\}$ and a $p$-proper $L$-labelling $\ell$ of $G'$ matching $\ell'$ on all edges but possibly $u_{i+1}u'_{i+1}$, and such that $\abs{\ell(u_{i+1}u'_{i+1}) \lambda_i} \neq 1$.}

The definition of $\ell$ and the fact previous Cases~3 and~4 do not apply, imply that $A'_{i+1}$, $A'_{i+2}$ and $a'_{i+2}$ are the same by both $\ell'$ and $\ell$. From Identity~\ref{eq:subcubic-3}, we deduce that $L_i = \ell(u_{i+1}u'_{i+1}) \lambda_i L_i$, where $\lambda_i$ is the same by both $\ell'$ and $\ell$. Now consider $x_0 \in L_i$; from what we have just deduced, we now get that $$ \set{(\ell(u_{i+1}u'_{i+1})\lambda_i)^j x_0 }_{j \in \NN} \subseteq L_i.$$
Because $\abs{\ell(u_{i+1}u'_{i+1}) \lambda_i} \neq 1$,
we then deduce that the set $\set{(\ell(u_{i+1}u'_{i+1})\lambda_i)^j x_0 }_{j \in \NN}$  has infinite cardinality and is included in $L_i$, which has size~$4$; a contradiction.
\end{enumerate}

Note that, by $\ell'$, there are actually at least two values in $L(u_iu'_i)$ that could be assigned to $u_iu_i'$ without breaking p-properness. This is because $|L(u_iu_i')|=4$, and, when labelling $u_iu_i'$, we only have to make sure that $u_i'$ gets product different from that of its at most two neighbours different from $u_i$ in $G'$ (in particular, note that we must have $A_i' \neq 1$ by $\ell'$ so that $\pi_{\ell'}(u_i) \neq \pi_{\ell'}(u_i')$, and thus we do not have to care about $u_i$ and $u_i'$ getting in conflict when relabelling $u_iu_i'$). 
Because Case~5 does not apply, this actually implies that there are exactly two such values from every $L(u_iu_i')$, and that these two values are precisely $a_i$ and $-a_i$.


\begin{enumerate}[resume]
\item \emph{There exists $i \in \{1,\dots,p\}$ such that $L_i \neq \set{\alpha,-\alpha,\beta,-\beta}$ for some distinct $\alpha,\beta \in \mathbb{R}^*$.}

Let us consider the identity $L_i = a'_{i+1}\lambda_i  L_{i}$ again. Since Case~5 does not apply, we have $\abs{\ell'(u_{i+1}u'_{i+1})\lambda_i} = 1$ for any possible value as $\ell'(u_{i+1}u'_{i+1})$ from $L(u_{i+1}u_{i+1}')$. Since $u'_{i+1}$ has, in $G'$, two neighbours different from $u_{i+1}$, there are, in $L(u_{i+1}u_{i+1}')$, two possible values for $u_{i+1}u'_{i+1}$ that make $u_{i+1}'$ being not in conflict with these two neighbours, and these at least two possibilities must include $a'_{i+1}$ and $-a'_{i+1}$. Now, by considering the p-proper $L'$-labelling of $G'$ obtained from $\ell'$ by changing the label of $u_{i+1}u'_{i+1}$ to $-a_{i+1}$, the same reasoning process leads us to deduce that $L_i = -a'_{i+1}\lambda_i  L_{i}$. This implies that $L_i = - L_i$, a contradiction.
\end{enumerate}

We are now ready to conclude the proof, by considering a few cases on the length of $C$. The crucial points to keep in mind from now on, are that $L$ verifies, for every $i \in \{1, \dots, p\}$, that 1) $a_i',-a_i' \in L(u_iu_i')$ and, in $\ell'$, changing the label of $u_iu_i'$ from $a_i'$ to $-a_i'$ cannot raise a conflict in $G'$, and that 2) there are nonzero real numbers $\alpha_i,\beta_i$ such that $L_i = \{\alpha_i,-\alpha_i,\beta_i,-\beta_i\}$.

\begin{enumerate}[resume]
\item \emph{$p$ is even.} 

For every $i \in \{1,\dots,p\}$, we associate a variable $x_i$ to the edge $u_iu_{i+1}$. We consider the polynomial $$P(x_1, \dots, x_p)= \prod_{i=1}^p \left(x_{i-1}x_i - A_i'\right),$$
which is equivalent to considering $$P'(y_1,\dots,y_p)=\prod_{i=1}^p \left(y_{i-1} + y_i - \log(A_i')\right)$$ where $y_i = \log x_i$ for every $i \in \{1,\dots,p\}$. Note that the monomial $y_1\dots y_p$ has maximum degree and nonzero coefficient in the expansion of $P'$. Thus, by the Combinatorial Nullstellensatz, we can assign values to the $y_i$'s so that $P'$ does not vanish, assuming we have at least two possible values to choose from for each of the $y_i$'s. This implies that we can assign values to the $x_i$'s so that $P$ does not vanish, assuming we have at least two possible values with distinct absolute values to choose from, for each of the $x_i$'s. Particularly, since $|L(u_iu_{i+1})|=4$ for every edge $u_iu_{i+1}$, this implies that $\ell'$ can be extended to the edges of $C$, resulting in an $L$-labelling $\ell$ of $G$ where $\pi_{\ell}(u_i)$ and $\pi_{\ell}(u_i')$ have distinct absolute values for every $i \in \{1,\dots,p\}$. Now, the only possible remaining conflicts are between the $u_i$'s. Due to all the assumptions made this far, recall, for every $i \in \{1,\dots,p\}$, that $\ell$ assigns label $a_i'$ to every edge $u_iu_i'$, that $-a_i' \in L(u_iu_i')$, and that switching $\ell(u_iu_i)$ from $a_i'$ to $-a_i'$ cannot raise a conflict between $u_i'$ and its neighbours. Thus, to get a p-proper $L$-labelling of $G$, we can just consider each of the $u_iu_i'$'s in turn, and for each $u_iu_i'$ of them, switch, if necessary, its label to $-a_i'$ so that $u_i$ gets positive product if $i$ is, say, even, or negative product otherwise.

\item \emph{$p=3$.} 

Because Cases~3 and~4 do not apply, recall that $u_1',u_2',u_3'$ are pairwise different. We extend $\ell'$ as follows. We start by assigning any label from $L(u_1u_2)$ to $u_1u_2$. Next, we assign to $u_3u_1$ a label from $L(u_3u_1)$ so that no conflict between $u_1$ and $u_1'$ arises, and the resulting partial products of $u_2$ and $u_3$ have different absolute values. Note that this is possible, since $L_3$ is of the form $\{\alpha,-\alpha,\beta,-\beta\}$. We finally assign to $u_2u_3$ a label from $L(u_2u_3)$ so that there is no conflict between $u_2$ and $u_2'$, $u_3$ and $u_3'$, and $u_1$ and $u_3$. Recall that $u_2$ and $u_3$ cannot be in conflict due to how $u_3u_1$ was labelled. Thus, the only potential conflict that can remain is between $u_2$ and $u_1$, and, if it occurs, then we can get rid of it by simply changing the label of $u_2u_2'$ from $a_2'$ to $-a_2'$. Recall that this cannot make $u_2'$ get in conflict with its neighbours different from $u_2$, and that $u_2$ and $u_2'$ also cannot get in conflict unless they already were before switching the label of $u_2u_2'$.

\item \emph{$p$ is odd at least~$5$.}

We first use the Combinatorial Nullstellensatz similarly as in Case~7, to label the edges of $C$ in such a way that, for certain pairs of vertices, the resulting products have distinct absolute values. More precisely, we want to achieve this for the pairs $\{u_1,u_1'\}$, $\{u_1,u_2\}$, $\{u_2,u_3\}$, $\{u_3,u_3'\}$, $\{u_4,u_4'\}$, $\{u_5,u_5'\}$, $\dots$, $\{u_{p-2},u_{p-2}'\}$ and $\{u_p,u_p'\}$. We denote by $\mathcal{S}$ the set of those pairs. In order to show that such an extension exists, for every $i \in \{1,\dots,p\}$ we associate a variable $x_i$ to the edge $u_iu_{i+1}$, and consider the polynomial
\begin{equation*}
\begin{split}
    P(x_1, \dots, x_p) = & \left(x_px_1 - A_1' \right) \cdot \left(x_pa_1' - x_2a_2' \right) \cdot \left(x_1a_2' - x_3a_3' \right) 
    \\ & \cdot \left(\prod_{i=3}^{p-2} \left(x_{i-1}x_i - A_i' \right)\right) \cdot \left(x_{p-1}x_p - A_p' \right),
\end{split}
\end{equation*}

which, if $y_i=\log \abs{x_i}$ for every $i \in \{1,\dots,p\}$,
is the same as considering
\begin{equation*}
\begin{split}
P'(y_1, \dots, y_p)= &\left(y_p + y_1 - \log(A_1') \right) \cdot \left(y_p + \log(a_1') - y_2 - \log(a_2') \right) \cdot \left(y_1 + \log(a_2') - y_3 - \log(a_3') \right) \\ & 
\cdot \left(\prod_{i=3}^{p-2} \left(y_{i-1} + y_i - \log(A_i') \right)\right) \cdot \left(y_{p-1} + y_p - \log(A_p') \right).
\end{split}
\end{equation*}
It can be checked that, in the expansion of $P'$, 
the monomial $y_1\dots y_p$ has maximum degree and nonzero coefficient $-2$.
Thus, by the Combinatorial Nullstellensatz we deduce that there is a way to label the edges of $C$ with labels from their respectives lists, so that the desired conflicts (between the adjacent vertices in the pairs of $\mathcal{S}$) are avoided. In particular, this is possible because all these lists are of the form $\{\alpha,-\alpha,\beta,-\beta\}$, and, in particular, contain two values with distinct absolute values.

The resulting labelling might be not p-proper, and, to turn it into a p-proper one, we will \textit{switch} some edges incident to the vertices in $C$, and, by that, we mean changing the current label $l$ of an edge to $-l$. More particularly, we will switch edges of the form $u_iu_{i+1}$ and $u_iu_i'$; due to some of the assumptions made this far, recall that for every such edge $e$ with current label $l$, we do have $-l \in L(e)$.

We start by switching, if necessary, $u_2u_2'$ and $u_{p-1}u_{p-1}'$ so that the products of $u_2'$ and $u_{p-1}'$ get positive and negative, respectively. Next, we switch $u_1u_2$, if necessary, so that the product of $u_2$ gets negative. Now, we consider the edges $u_3u_4, u_4u_5, \dots, u_pu_1$ one by one following this ordering, and, for every such considered edge $u_iu_{i+1}$, we switch it, if necessary, so that the product of $u_i$ gets negative if $i$ is odd, and positive otherwise. Lastly, we switch $u_1u_1'$, if necessary, so that the product of $u_1$ gets negative.

We claim that the eventual labelling of $G$ is p-proper, our final contradiction. First recall, as mentioned earlier, that the switching operation guarantees that the resulting labelling is an $L$-labelling. Its p-properness follows from the following arguments. First, for all the pairs of adjacent vertices in $\mathcal{S}$, the products are different due to distinct absolute values (preserved under the switching operation). Regarding the two adjacent vertices in the pair $\{u_{p-1},u_{p-1}'\}$, the products have different signs and are thus different. Now, for every two adjacent vertices in the pairs $\{u_3,u_4\}, \{u_4,u_5\}, \dots, \{u_p,u_1\}$, the products are different due to their signs being different.
\qedhere
\end{enumerate} 
%
%
%
%

%
%
%
%
\end{proof}

\section{Conclusion}

In this work, we have considered a problem being a combination of the Multiplicative 1-2-3 Conjecture and of the List 1-2-3 Conjecture, standing as a List Multiplicative 1-2-3 Conjecture. In particular, we have exhibited a few bounds on the parameter $\chps$, both for graphs in general and for more specific classes of graphs. While some of these bounds are tight, some others remain a bit distant from what we believe should be optimal.

An interesting point stemming from our proofs, is the methods we have used to establish our bounds. In the context of the List 1-2-3 Conjecture, the algebraic approach, through, in particular, the polynomial method and tools such as the Combinatorial Nullstellensatz, is definitely the best approach we know of at the moment to establish bounds on $\chs$. As described notably in Subsection~\ref{subsection:algebraic-tools}, and seen throughout this work, the potential of this method is a bit less obvious for exhibiting bounds on $\chps$.  Recall that, in the current work, we have mainly exploited the connection between $\chs$ and $\chps$ established in Theorem~\ref{theorem:chs-to-chp}. It might be, however, that there are dedicated ways to better exploit the algebraic approach, and get better bounds on $\chps$.

As a main perspective for further work on the topic, it would be nice to obtain a constant upper bound on $\chps$ for graphs in general. Recall that, due to Theorem~\ref{theorem:chs-to-chp}, this could be obtained through establishing a constant upper bound on $\chs$. This apart, it would be interesting to verify the List Multiplicative 1-2-3 Conjecture for more classes of graphs. For instance, it would be interesting to improve any of the upper bounds in Corollary~\ref{corollary:existing-bounds}, some of which we have already improved in Subsection~\ref{subsection:bounds-particular-classes}. Notably, it is worth mentioning that the arguments used to prove Theorems~\ref{theorem:planar-girth16} and~\ref{theorem:subcubic} are tight, and, as a result, it seems that our proofs would be hard to improve to lower the bound of $4$. From this, we would be interested in having a proof of the List Multiplicative 1-2-3 Conjecture for planar graphs with girth at least~$16$ or for subcubic graphs.

\end{document}